\documentclass[sn-mathphys]{sn-jnl}

\jyear{2021}%
\newcommand{\rset}{\mathbb{R}}
\DeclareMathOperator*{\argmin}{arg\,min}
\providecommand{\norm}[1]{\lVert#1\rVert}

\usepackage{amsmath}
\usepackage{amssymb}
\usepackage{amsthm}
\usepackage{graphicx}
\usepackage{comment}
\theoremstyle{thmstyleone}%
\newtheorem{theorem}{Theorem}
\theoremstyle{thmstyletwo}%
\newtheorem{example}{Example}%
\newtheorem{remark}{Remark}%
\newtheorem{lemma}{Lemma}%
\newtheorem{assumption}{Assumption}%

\theoremstyle{thmstylethree}%
\newtheorem{definition}{Definition}%

\usepackage{color}
\newcommand{\red}{\textcolor{black}}
\newcommand{\blue}{\textcolor{black}}

\usepackage[symbol]{footmisc}

\begin{document}

\title[Y. Nabou and I. Necoara]{ Efficiency of higher-order algorithms for minimizing composite functions}


\author[1]{\fnm{Yassine} \sur{Nabou}}\email{yassine.nabou@stud.acs.upb.ro}
\author[1,2]{\fnm{Ion} \sur{Necoara}}\email{ion.necoara@upb.ro}

\affil[1]{\orgdiv{Automatic Control and Systems Engineering Department}, \orgname{National University of Science and Technology Politehnica Bucharest}, \orgaddress{\street{Spl. Independentei 313}, \postcode{060042}, \city{Bucharest}, \country{Romania}}}

\affil[2]{\orgdiv{Gheorghe Mihoc-Caius Iacob Institute of Mathematical Statistics and Applied Mathematics
of the Romanian Academy}, \postcode{050711}, \city{Bucharest}, \country{Romania}}

\abstract{Composite minimization involves a collection of  functions which are aggregated in a nonsmooth manner. It covers, as a particular case, smooth approximation of minimax games, minimization of max-type functions, and simple composite minimization problems, where the objective function has a nonsmooth component.  We design a higher-order majorization algorithmic framework for fully composite problems (possibly nonconvex). Our framework  replaces  each  component  with a higher-order surrogate such that the corresponding error function has a higher-order Lipschitz continuous derivative.  We present convergence guarantees for our method for composite optimization problems with (non)convex and (non)smooth objective function.  In particular,  we  prove  stationary point convergence  guarantees for general nonconvex (possibly nonsmooth) problems and under Kurdyka-Lojasiewicz (KL) property of the objective function we derive improved rates depending on the KL parameter. For convex (possibly nonsmooth) problems we also provide sublinear convergence rates.}


\maketitle

\section{Introduction}\label{sec1} 
In this work, we consider the class of general composite optimization problems:  

\vspace{-0.3cm}

\begin{equation}\label{eq:optpb}
	\min_{x \in \text{dom}f} f(x):= g\big(F(x)\big)+h(x), 
\end{equation}

\vspace{-0.0cm}

\noindent where  $h:\mathbb{E}\to \bar{\rset}$, with $\bar{\rset} = \rset \cup \{+ \infty\}$, and  $F: \mathbb{E} \to \bar{\mathbb{R}}^{m}$ are general proper lower semicontinuous functions on their closed domains and $g:\rset^{m} \to \rset$ is a proper closed convex function defined everywhere. Here, $\mathbb{E}$ is a finite-dimensional real vector space and $F=(F_1, \cdots, F_m)$.  Note  that $\text{dom}\,f =g(\text{dom}F)\cap\text{dom}\,h $. 
This formulation unifies many particular cases, such as smooth approximation of minimax games,  max-type minimization problems or exact penalty formulations of nonlinear programs, while recent instances include robust phase retrieval  and matrix factorization problems \cite{BolChe:20,DoiNes:21,DruPaq:19,LiNg:07}.  Note that the setting where $g$ is the identity function was intensively investigated in  large-scale optimization \cite{AttBol:09,Mai:15,Nes:13,NecNes:16}.  In this paper, we call this formulation \textit{simple} composite optimization.  When $g$ is restricted to be a Lipschitz convex function and $F$ smooth, a natural approach to this problem consists in linearizing the smooth part, leaving the nonsmooth term unchanged and adding an appropriate quadratic regularization term.  This is the approach considered e.g., in \cite{DruPaq:19,Pau:16}, leading to a proximal Gauss-Newton method, i.e.  given the current point  $\bar{x}$ and a regularization parameter $t>0$, solve at each iteration the subproblem:

\vspace{-0.6cm}

\begin{align*}
	x^{+}=\argmin\limits_{x} g\Big(F(\bar{x}) + \nabla F(\bar{x})(x-\bar{x})\Big) + \frac{1}{2t}\norm{x-\bar{x}}^{2} + h(x).
\end{align*}

\vspace{-0.3cm}

\noindent For such a method it was proved in \cite{DruPaq:19} that  $\text{dist}(0,\partial f(x))$ converges to $0$ at a sublinear rate of order $\mathcal{O}(1/k^{\frac{1}{2}})$, where $k$ is the iteration counter, while  convergence of the iterates under KL inequality was recently shown in \cite{Pau:16}. Note that the case where $g$ is (Lipschitz) convex, $F$ is smooth and $h = 0$ has been also analysed in (\cite{Yua:85}) \cite{Fle:82}.  
%
%
%
\noindent In \cite{BolChe:20} a flexible method is proposed, where the smooth part $F$ is replaced  by its quadratic approximation,  i.e.,  given  $\bar{x}$, solve:
\vspace{-0.6cm}

\begin{align*}
	x^{+}=\argmin\limits_{x} g\left( F(\bar{x}) +\nabla F(\bar{x})(x-\bar{x}) + \frac{L}{2}\norm{x-\bar{x}}^{2} \right) + h(x), 
\end{align*}

\vspace{-0.3cm}

\noindent where  $L=(L_1,\cdots,L_m)^T$, with $L_i$ being  the Lipschitz constant of the gradient of $F_i$, for $i=1:m$. Assuming $F$, $g$ and $h$ are convex functions, and $g$ additionally is componentwise nondecreasing and Lipschitz, \cite{BolChe:20}  derives  sublinear convergence rate  of order  $\mathcal{O}(1/k)$ in function values. Finally,  in the  recent paper \cite{DoiNes:21},  a general composite minimization problem of the form:

\vspace{-0.6cm}

\begin{align*}
	\min\limits_{x\in\mathbb{E}}g\big(x,F(x)\big),
\end{align*}

\vspace{-0.3cm}

\noindent is considered, where $F=(F_{1},\cdots,F_{m})$, with $F_{i}$'s being convex and $p$-smooth functions on $\mathbb{E}$ and having the $p$-derivative Lipschitz, with $p\geq 1$ an integer constant. Under  these settings,  \cite{DoiNes:21} replaces the smooth part by its  Taylor approximation of order $p$ plus a proper regularization term, i.e., given  $\bar{x}$, solve the following subproblem: 

\vspace{-0.0cm}

\begin{align*}
	x^{+}=\argmin\limits_{x}g\left( x, T^{F}_{p}(x;\!\!\bar{x}) + \frac{L}{(p+1)!}\norm{x-\bar{x}}^{p+1}\right),
\end{align*}

\vspace{-0.3cm}

\noindent where $L=(L_1,\cdots,L_m)^T$, with $L_i$ being related to  the Lipchitz constant of the $p$-derivative of  $F_i$ and $T^{F}_{p}(x;\!\!\bar{x}) $ is the $p$-Taylor approximation of $F$ around the current point $\bar{x}$.  For such a higher-order method, in the convex settings  and assuming that $g$ has full domain in the second argument, \cite{DoiNes:21} derives a sublinear convergence rate in function values of order $\mathcal{O}\left(1/k^p\right)$. 

\medskip

\noindent Note that the optimization scheme in \cite{DoiNes:21} belongs to the class of  \textit{higher-order methods}. Such methods are popular due to their performance in dealing with ill conditioning and fast rates of convergence, see e.g.,    \cite{BriGar:17,carGo17,GraNes:19,Nes:20,Nes:20Inxt,NecDan:20}.   For example,  first-order methods achieve convergence rates of order $\mathcal{O}(1/k)$ for smooth convex optimization \cite{Mai:15,Nes:13}, while   higher-order methods of order $p >1$ have converge rates $\mathcal{O}(1/k^p)$  for minimizing  $p$ smooth convex objective functions \cite{GraNes:19,Nes:20,Nes:20Inxt,NecDan:20}. Accelerated variants of higher-order methods  were also developed e.g., in \cite{Gas:19,Nes:20,Nes:20Inxt}.  Recently, \cite{NecDan:20} provided a unified framework for the convergence  analysis of higher-order optimization algorithms for solving simple  composite optimization problems using the \textit{majorization-minimization} approach. This is a technique that approximate an objective function by a majorization function, which can be minimized in closed form or its solution computed fast, yielding a solution or some acceptable improvement.  Note that papers such as \cite{BolChe:20,Mai:15} use a first-order majorization-minimization approach to build a model (i.e., use only  gradient information), while \cite{NecDan:20} uses higher-order derivatives to build such a model. \textit{However, global complexity bounds for higher-order methods based on the majorization-minimization principle for  solving composite  problem \eqref{eq:optpb} are not yet given.}  This is the goal of this work.

\medskip

\noindent\textbf{Contributions.} In this paper, we provide an algorithmic framework based on the notion of higher-order upper bound approximation of the composite problem \eqref{eq:optpb}. Note that in this optimization formulation we consider general properties for our objects, e.g.,  the functions  $F$ and $h$ can be smooth or nonsmooh, convex or nonconvex and $g$ is convex, nondecreasing  and has full domain.  Our framework  {consists} of replacing $F$ by a higher-order surrogate,  leading to a \textit{General Composite Higher-Order} minimization algorithm, which we call GCHO. This majorization minimization approach is relevant as it yields an array of algorithms, each of which is associated with the specific properties  of $F$ and the corresponding surrogate, and it provides a unified convergence analysis.  Note that most of our variants of GCHO  for $p> 1$ were not explicitly considered in the literature before (at least in the nonconvex settings). 
Moreover, our new first-, second-, and third-order methods can be implemented in practice using existing efficient  techniques from  e.g., \cite{Nes:20Inxt,WacBie:06}.  

\medskip 

\noindent We derive convergence guarantees for the GCHO algorithm when the upper bound approximate $F$ from the objective function up to an error that is $p \geq 1$ times differentiable and has a Lipschitz continuous $p$ derivative; we call such upper bounds composite higher-order surrogate functions. More precisely,  on composite (possibly nonsmooth) nonconvex problems we prove for GCHO,  with the help of a  new auxiliary sequence,   convergence rates $\mathcal{O}\left(\frac{1}{k^{p/(p+1)}}\right)$ in terms of first-order optimality conditions. We also characterize the convergence rate of GCHO algorithm locally, in terms of function values, under the Kurdyka-Lojasiewicz (KL) property. Our result show that the convergence behavior of GCHO ranges from sublinear to linear depending on the parameter of the underlying KL geometry. Moreover, on general (possibly nonsmooth) composite convex problems (i.e.,  $f$ is convex function) our algorithm achieves global sublinear convergence rate of order $\mathcal{O}\left(1/k^{p}\right)$ in function values. We summaries our convergence results in Table \ref{Tabl:result}. Finally, for $p=2$, $g(\cdot) = \max(\cdot)$ \blue{and $h = 0$}, we show that the subproblem, even in the nonconvex case, is equivalent to minimizing an explicitly written convex function over a convex set that can be solve using efficient convex optimization tools.

\medskip

\noindent Besides providing a general framework for the design and analysis of composite higher-order methods, in special cases, where complexity bounds are known for some particular  algorithms, our convergence results recover the existing bounds.  For example, from our convergence analysis one can easily  recover the convergence bounds  of higher-order algorithms  from \cite{Nes:20} for unconstrained minimization and  from  \cite{Nes:20,Nes:20Inxt,NecDan:20} for simple composite minimization. Furthermore,  in the composite \textit{convex} case we recover the convergence bounds from  \cite{BolChe:20} for $p=1$ and particular choices of $g$ and  from   \cite{DoiNes:21} for $p \geq 1$. To the best of our knowledge, this is the first complete work  dealing with composite problems  in the nonconvex and nonsmooth settings,  and  explicitly deriving convergence bounds for higher-order majorization-minimization algorithms  (including  convergence under KL). 

\begin{table}
\caption{Convergence results for GCHO algorithm.}
\begin{center}
\begin{tabular}{|l|l|l|l|}
\hline
\multirow{2}{*}{nonconvex case} & Ass. \ref{ass:fun} and \ref{cond:2}    & \begin{tabular}[c]{@{}l@{}}convergence rate: \\ $\exists (y_k)_{k\geq 0}$ close to $(x_k)_{k\geq 0}$ s.t. \\ $\min\limits_{j=0:k} \text{dist}(0, \partial f(y_{j}))\leq \mathcal{O}\left(k^{-\frac{p}{p+1}}\right)$\end{tabular} & Theorem \ref{th:MHO} \\ \cline{2-4} 
                                 & \begin{tabular}[c]{@{}l@{}} Ass. \ref{ass:fun},\ref{cond:2}, \ref{cond:3} \& KL \end{tabular} & \begin{tabular}[c]{@{}l@{}}convergence rate:\\$f(x_{k})\to {f_*}$ sublinear or linear \end{tabular} & Theorem \ref{Thm:kl} \\ \hline
convex case                      & \begin{tabular}[c]{@{}l@{}} Ass. \ref{ass:fun} \& \\   $f$ convex \end{tabular}         & \begin{tabular}[c]{@{}l@{}}convergence rate:\\$f(x_k) - f^*\leq \mathcal{O}\left(k^{-p}\right)$ \end{tabular}  & Theorem \ref{Theorem:convex} \\ \hline
\end{tabular}
\label{Tabl:result}
\end{center}
\end{table}

\medskip 



\section{Notations and  preliminaries}\label{sec2}
We use the standard notations from \cite{Nes:20,Nes:20Inxt}. We denote a finite-dimensional real vector space with $\mathbb{E}$ and $\mathbb{E}^{*}$ its dual space composed of linear functions on $\mathbb{E}$. For any linear function $s\in\mathbb{E}^{*}$, the value of $s$ at point $x\in\mathbb{E}$ is denoted by $\langle s,x\rangle$.  Using a self-adjoint positive-definite operator ${B}:\mathbb{E}\rightarrow \mathbb{E}^{*}$, we endow these spaces with conjugate Euclidean~norms:

\begin{align*}
\norm{x}=\langle {B}x,x\rangle^{{\frac{1}{2}}},\quad x\in \mathbb{E},\qquad \norm{g}_{*}=\langle g,{B}^{-1}g\rangle^{\frac{1}{2}},\quad g\in \mathbb{E}^{*}.
\end{align*}

\noindent For a twice differentiable function $\phi$ on a convex and open domain $\text{dom}\,\phi \subseteq \mathbb{E}$, we denote by $\nabla \phi(x)$ and $\nabla^{2} \phi(x)$ its gradient and hessian evaluated at  $x\in \text{dom}\,\phi$, respectively. Then, $\nabla \phi(x)\in\mathbb{E}^{*}$ and $\nabla^{2}\phi(x)h\in\mathbb{E}^{*}$ for all $x\in\text{dom}\,\phi$, $h\in\mathbb{E}$.  Throughout the paper, we consider $p$ a positive integer.  In what follows, we often work with directional derivatives of function $\phi$ at $x$ along directions $h_{i}\in \mathbb{E}$ of order $p$, ${D^{p}} \phi (x)[h_{1},\cdots,h_{p}]$, with $i=1:p$.  If all the directions $h_{1},\cdots,h_{p}$ are the same, we use the notation
$ {D^{p}} \phi(x)[h]$, for $h\in\mathbb{E}.$
Note that ${D^{p}} \phi (x)$ is a symmetric $p$-linear form. Then, its norm is defined as:

\vspace{-0.6cm}

\begin{align*}
\norm{{D^{p}} \phi (x)}= \max\limits_{h\in\mathbb{E}} \left\lbrace {D^{p}} \phi (x)[h]^{p} :\norm{h}\leq 1 \right\rbrace.     
\end{align*}

\vspace{-0.4cm}

\noindent Further, the Taylor approximation of order $p$ of $\phi$ at $x\in \text{dom}\,\phi$ is denoted:

\vspace{-0.7cm}

\begin{align*}
        T_p^{\phi}(y;\!\! x)= \phi(x) + \sum_{i=1}^{p} \frac{1}{i !} {D^{i}} \phi(x)[y-x]^{i} \quad \forall y \in \mathbb{E}.
\end{align*}
    
\vspace{-0.4cm}
    
\noindent If $\phi: \mathbb{E}\to \bar{\rset}$ is $p$ differentiable function on $\text{dom}\,\phi$, then the $p$th derivative is Lipschitz continuous if there exist a constant $L_p^{\phi} > 0$ such that:

\vspace{-0.2cm}

   	\begin{equation} \label{eq:001}
    	\| {D^p} \phi(x) - {D^p} \phi(y) \| \leq L_p^{\phi} \| x-y \| \quad  \forall x,y \in \text{dom}\,\phi.
   	\end{equation}

\noindent It is known that if \eqref{eq:001} holds,  then the residual between the function  and its Taylor approximation can be bounded \cite{Nes:20}:

\vspace{-0.4cm}
    \begin{equation}\label{eq:TayAppBound}
    \vert \phi(y) - T_p^{\phi}(y;\!\!x) \vert \leq  \frac{L_p^{\phi}}{(p+1)!} \norm{y-x}^{p+1} \quad  \forall x,y \in \text{dom}\,\phi.
    \end{equation}

\vspace{-0.05cm}
  
\noindent  If $p \geq 2$, we also have the following inequalities valid for all $ x,y \in \text{dom}\,\phi$:

\vspace{-0.7cm}

    \begin{align} \label{eq:TayAppG1}
    &\| \nabla \phi(y) - \nabla T_p^{\phi}(y;\!\!x) \|_* \leq \frac{L_p^{\phi}}{p!} \|y-x \|^p, \\
    \label{eq:TayAppG2}
    &\| \nabla^2 \phi(y) - \nabla^2 T_p^{\phi}(y;\!\!x) \| \leq \frac{L_p^{\phi}}{(p-1)!} \| y-x\|^{p-1}.
    \end{align}	
    
    \vspace{-0.3cm}

%
%
\noindent \noindent \red{For the Hessian, the norm defined in \eqref{eq:TayAppG2} corresponds to the spectral norm of self-adjoint linear operator (maximal module of all eigenvalues computed w.r.t. $B$)}. 
\noindent A function $g:\rset^{m} \to \rset$ is said to be nondecreasing if for all $i=1:m$, $g$ is nondecreasing in its $i$th argument, i.e., the univariate function:

\vspace{-0.8cm}

\begin{align*}
    z \mapsto g(z_{1},\cdots,z_{i-1},z,z_{i+1},\cdots,z_{m}),
\end{align*}

\vspace{-0.3cm}

\noindent is nondecreasing. In what follows, if $x$ and $y$ are in $\rset^{m}$, then $x\geq y$ means that $x_{i}\geq y_{i}$ for all $i=1:m$. Similarly, we define $x>y$. \red{Since $g$ is nondecreasing, then for all $x,y\in\mathbb{R}^m$ such that $x\leq y$ we have $g(x)\leq g(y)$}.  Next, we provide few definitions and properties concerning subdifferential calculs (see \cite{Mor:06}).

\vspace{-0.3cm}

\begin{definition}
\textbf{(Subdifferential)}: Let $f: \mathbb{E} \to \bar{\mathbb{R}}$ be a proper lower semicontinuous function. For a given $x \in \text{dom} \; f$, the regular subdifferential of $f$ at $x$, written $\widehat{\partial}f(x)$, is the set of all vectors $g_{x}\in\mathbb{E}^{*}$ satisfying:
\begin{equation*}
\lim_{x\neq y}\inf\limits_{y\to x}\frac{f(y) - f(x) - \langle g_{x}, y - x\rangle}{\norm{x-y}}\geq 0.
\end{equation*}
When $x \notin\text{dom} \; f$, we set $\widehat{\partial} f(x) = \emptyset$. The limiting-subdifferential, or simply the subdifferential, of $f$ at $x\in \text{dom} \, f$, written $\partial f(x)$, is defined as \cite{Mor:06}:

\vspace{-0.3cm}

\begin{align*}
\partial f(x):= \left\{ g_{x}\in \mathbb{E}^{*}\!\!: \exists x^{k}\to x, f(x^{k})\to f(x) \; \text{and} \; \exists g_{x}^{k}\in\widehat{\partial} f(x^{k}) \;\; \text{such that} \;\;  g_{x}^{k} \to g_{x}\right\}. 
\end{align*} 
\end{definition}

\vspace{-0.6cm}

\noindent Note that we have $\widehat{\partial}f(x)\subseteq\partial f(x)$  for each $x\in\text{dom}\,f$. In the previous inclusion, the first set is closed and convex while the second one is closed, see e.g.,   \cite{Mor:06}. { For the composite function $g(F)$, where $g$ is proper lower semicontinuous, convex and regular at $F(x)$, and $F=(F_{1},\cdots,F_{m})$, with $F_{i}$'s differentiable, we have the following chain rule: 
\begin{align}\label{eq:ch_rl}
\partial (g\circ F)(x) =  \nabla F(x)^T \partial g(F(x)).
\end{align} 
\noindent Further, given $f = f_1 + f_2$, where $f_i$'s are proper lsc functions either both convex or $f_1$ locally Lipschitz, then for all $x\in\text{dom}\, f$ we have: 
$$\partial f(x)\subseteq \partial f_1(x) + \partial f_2(x).$$
\noindent The equality holds e.g., if each convex $f_i$, for $i=1:2$, is also regular at $x$ or if $f_1$ is strictly differentiable.}
For any $x\in \text{dom} \; f$ let us define:

\vspace{-0.6cm}

\begin{align*}
S_{f}(x)=\text{dist}\big(0,\partial f(x)\big):=\inf\limits_{g_{x}\in\partial f(x)}\norm{g_{x}}.
\end{align*}

\vspace{-0.4cm}

\noindent If $\partial f(x) = \emptyset$, we set $S_f(x) = \infty$. 



\section{General composite higher-order algorithm}\label{sec3}
In this section, we propose a higher-order algorithm for solving the general composite problem \eqref{eq:optpb} and analyse its convergence.

\vspace{-0.3cm}

\begin{assumption}\label{ass:fun}
\red{We consider the following assumptions for optimization problem \eqref{eq:optpb}:}
\begin{enumerate}
\item[1.] The functions $F_{i}$, with $i=1\!:\!m$, $g$ and $h$ are proper lower semicontinuous on their domains, satisfy the chain rule \eqref{eq:ch_rl} \red{and $\text{dom}\, h \subseteq g(\text{dom}\,F)$.} 
\item[2.] Additionally, $g$ is convex nondecreasing and has full domain, satisfying the following subhomogeneity property:

\vspace{-0.3cm}

\begin{equation}\label{eq:g}
 g(\alpha x)\leq\alpha g(x)\quad\forall x\in \mathbb{R}^m \quad \forall\alpha \geq 1.
\end{equation}

\item[3.]  Problem \eqref{eq:optpb} has a solution and thus $f^{*}:=\!\inf_{x\in\text{dom}\,f}\! f(x)\!>\!-\infty$. 
\end{enumerate}
\end{assumption}

\vspace{-0.3cm}

\noindent \red{From Assumption \ref{ass:fun}.1, it follows that $\text{dom}\,f = \text{dom}\,h$}. Moreover, if Assumption \ref{ass:fun}.2 holds, then from  \cite{DoiNes:21}(Theorem 4) it follows that:

\vspace{-0.2cm}

\begin{equation}\label{Lemma:1}
g(x+ty)\leq g(x) + tg(y) \quad \forall t\geq 0.
\end{equation}


$
$

\vspace{-0.3cm}



\noindent Next, we provide several examples of optimization problems that can be written as \eqref{eq:optpb} and satisfy our Assumption~\ref{ass:fun}.

\medskip

\vspace{-0.6cm}

\begin{example}\label{expl:3}(Minimax strategies for nonlinear games)
Let us consider the  problem:
\begin{align*}
\min_{x\in\triangle_n} \left\lbrace f(x): = \max_{u\in\triangle_m}\langle F(x),u\rangle \right\rbrace,
\end{align*}
where $\triangle_n$, $\triangle_m$ are the standard simplexes  in $\mathbb{R}^n$ and $\mathbb{R}^m$, respectively. The smooth approximation for this problem using the entropy distance is as follows \cite{Nes:05}:
\begin{align*}
 \min_{x\in\triangle_n} \left\lbrace f_\mu(x) :=  \max_{u\in\triangle_m}\Big\lbrace \langle F(x),u\rangle  - \mu\sum_{j=1}^{m}u_j\ln(u_j) - \mu\ln(m)\Big\rbrace\right\rbrace,
\end{align*}
for some $\mu>0$. Using Lemma 4 in \cite{Nes:05}, we get:
\begin{align*}
f_\mu(x) = \mu\ln\left(\sum_{j=1}^{m}e^{\frac{F_i(x)}{\mu}}\right).
\end{align*}
\noindent Hence,  considering $g(y) = \mu\ln\left(\sum_{j=1}^{m}e^{\frac{y_i}{\mu}}\right)$, then the original minimax problem can be approximated, for sufficiently small $\mu$,  with $\min_{x\in\Delta_n}f_\mu(x) := g(F(x))$. Note that $g$ satisfies Assumption \ref{ass:fun}.2.
\end{example}

\vspace{-1cm}

\begin{example}
\label{expl:4}(Min-max problems)
Let us consider the following min-max problem:
\begin{align*}
\min\limits_{x\in Q} \max\limits_{i=1:m}F_{i}(x).
\end{align*}
This type of problem is  classical in optimization. Note that if we define $g(y)=\max_{i=1:m}y_{i}$ and $h=\textbf{1}_{Q}$, then, the previous min-max problem can be written as problem \eqref{eq:optpb} and $g$ satisfies Assumption \ref{ass:fun}.2.
\end{example}
\vspace{-0.8cm}

\begin{example}
\label{expl:5}(Simple composite problems)
Let us consider the following simple composite minimization problem:

\vspace{-0.6cm}

\begin{align*}
\min\limits_{x\in\mathbb{R}^{n}} F_{0}(x)+h(x).
\end{align*} 

\vspace{-0.2cm}

\noindent Taking $F(x)=F_{0}(x)$ and $g$ the identity function, we can clearly see that $g\big(F(x)\big) + h(x)=F_{0}(x)+h(x)$. 
\end{example}

\vspace{-0.3cm}

\noindent Further, let us introduce  the  notion of a higher-order surrogate, see also   \cite{NecDan:20}.

\vspace{-0.4cm}

\begin{definition}\label{def:surg}
Let $\phi:\mathbb{E}\to\bar{\mathbb{R}}$ be a proper lower semicontinuous function and $x\in\text{dom}\,\phi$.  We call the function $s(\cdot\; ;\!x):  \mathbb{E} \to\bar{\mathbb{R}}$, with $\text{dom}\,  s(\cdot\; ;\!x)=\text{dom}\,\phi$, a $p$ higher-order surrogate of $\phi$ at $x$ if it has the following properties:
\item[(i)] the error function  
\begin{align}\label{sur:p2}
e(y;\!x)=s(y;\!x) - \phi(y), \; \text{with} \; y \in \text{dom}\ \phi, 
\end{align}
\noindent  is $p$ times differentiable  and the $p$ derivative is Lipschitz continuous with the Lipschitz constant $L_{p}^{e}$ on $\text{dom}\,\phi$.





\item[(ii)] the derivatives of the error function $e$ satisfy
\begin{align}\label{sur:p3}
\nabla^{i}e(x;\!\!x)=0 \quad \forall i=0:p, \; x\in\text{dom}\,\phi,
\end{align}  
and there exist a positive constant $R_{p}^{e}>0$ such that
\begin{align}\label{sur:p4}
e(y;\!\!x)\geq \frac{R_{p}^{e}}{(p+1)!}\norm{y-x}^{p+1}  \quad \forall x, y \in\text{dom}\,\phi.
\end{align}

\end{definition}

\vspace{-0.5cm}

\noindent Note that $\text{dom}\,e$ may be larger than  $\text{dom}\,\phi$ (in fact, in our examples below $\text{dom}\,e$ is equal to the whole space $\mathbb{E}$) and from \eqref{sur:p4} we have $s(y;x) \geq \phi(y)$ for all $x,y\in\text{dom}\,\phi$. Next, we give two nontrivial examples of higher-order surrogate functions, see  \cite{NecDan:20} for more examples.

\vspace{-0.4cm}

\begin{example}(Composite functions)
\label{exam:1}
Let $F_1:\mathbb{E}\to \mathbb{R}$ be $p$ times differentiable and the $p$ derivative be Lipschitz with constant $L_p^{\red{F_1}}$ and let $F_2:\mathbb{E} \to \bar{\mathbb{R}}$ be a proper closed function. Then, for the composite function $F=F_{1} + F_{2}$, where $\text{dom}\,F= \text{dom}\,F_2 $, one can consider the following $p$ higher-order surrogate function:
\begin{align*}
\red{s(y;\!\!x)}=T^{F_{1}}_{p}(y;\!\!x) + \frac{M_{p}}{(p+1)!}\norm{x-y}^{p+1} + F_2(y)\;\;\forall\;x,y\in\text{dom}\,F,
\end{align*}  

\vspace{-0.2cm}

\noindent where $M_{p}>L_{p}^{F_{1}}$.  Indeed, from the definition of the error function, we get:

\vspace{-0.4cm}

\begin{align}\label{eq:01}
\red{e(y;\!\!x)} =T^{F_{1}}_{p}(y;\!\!x) - F_{1}(y) + \frac{M_{p}}{(p+1)!}\norm{x-y}^{p+1}. 
\end{align}

\vspace{-0.2cm}

\noindent Thus $e(\cdot;\!x)$, with $\text{dom}\,e = \mathbb{E}$ and $\text{dom}\,F\subseteq \text{dom}\,e$, has the $p$ derivative  Lipschitz \red{with constant $L_p^{F_1} + M_p$ on $\text{dom}\,F$. Further, from the definition of the error function $e$, we have:
\begin{align*}
    \nabla ^{i} e(x;\!\!x) &= \nabla T^{F_1}_p(x;\!\!x) - \nabla^i F_1(x) = \nabla^i F_1(x) - \nabla^i F(x) = 0\;\; \forall i=1:p.
\end{align*}}
Moreover,  since $F_1$ has the $p$ derivative Lipschitz, it follows from \eqref{eq:TayAppBound} that:

\vspace{-0.4cm}

\begin{align*}   
T^{F_1}_{p}(y;\!\!x) - F_1 (y)\geq \frac{-L_{p}^{F_1}}{(p+1)!}\norm{x-y}^{p+1}.
 \end{align*}

 
\noindent Combining \red{this inequality with \eqref{eq:01}}, we get:
\begin{align}\label{eq:cd_adp}
\red{e(y;\!\!x)}\geq  \frac{M_{p} - L_{p}^{F_1}}{(p+1)!}\norm{x-y}^{p+1}.
\end{align}


\noindent Hence, the error function $e$ has $L_{p}^{e}=M_{p} + L^{F_1}_{p}$ and $R_{p}^{e}=M_{p}-L^{F_1}_{p}$.    
\end{example}

\vspace{-1cm}

\begin{example}(proximal higher-order)
\label{exam:2}	 
Let $F:\mathbb{E} \to \bar{\mathbb{R}}$ be a proper lower semicontinuous function. Then, we can consider the following higher-order surrogate function:
\begin{align*}
s(y;\!\!x)=F(y) + \frac{M_{r}}{(r+1)!}\norm{y-x}^{r+1},
\end{align*}

\vspace{-0.2cm}

\noindent where $r$ is a positive integer. Indeed, the error function is: 

\vspace{-0.3cm}

$$e(y;\!x)=s(y;\!\!x)-F(x)=\frac{M_{r}}{(r+1)!}\norm{y-x}^{r+1},$$

 \vspace{-0.1cm}

\noindent where $\text{dom}\,F\subseteq \text{dom}\,e = \mathbb{E}$. In this case, the error function $e$ has the $r$ derivative Lipschitz on $\text{dom}\,F$ with constants $L_{r}^{e}= M_{r}$ and $R_{r}^e=M_{r}$.
\end{example}

\vspace{-0.5cm}

\noindent In the following, we assume for problem \eqref{eq:optpb} that each function $F_{i}$, with $i=1:m$, admits a $p$ higher-order surrogate as in Definition \ref{def:surg}. Then, we propose the following General Composite Higher-Order algorithm, called GCHO.

\begin{center}
\noindent\fbox{%
\parbox{11.5cm}{%
\textbf{Algorithm GCHO}\\
Given $x_{0}\in\text{dom}\,f$. For $k\geq 1$ do:
\begin{enumerate}
\item[1.] Compute surrogate $s(x;\!\!x_{k})\!:=\big(s_{1}(x;\!\!x_{k}),\cdots,s_{m}(x;\!\!x_{k})\big)$ of $F$ near $x_{k}$.
\item[2.]Compute $x_{k+1} \in\text{dom}\,f$  satisfying the following descent:
\end{enumerate}
\begin{equation}\label{eq:loc}
g\big(s(x_{k+1};\!x_{k})\big) + h(x_{k+1})\leq f(x_{k}).
\end{equation}  
}%
}
\end{center}

\vspace{0.2 cm}
\noindent Although our algorithm requires that the next iterate  $x_{k+1}$ only  to satisfy the descent \eqref{eq:loc}, we usually generate $x_{k+1}$ by  solving the following subproblem:

\vspace{-0.6cm}

\begin{align}\label{eq:sub_prbm}
\min_{x} g\big(s(x;\!x_{k})\big) + h(x).
\end{align} 

\vspace{-0.3cm}

\noindent If $F$ and $h$ are convex functions, then the subproblem \eqref{eq:sub_prbm} can be also convex. Indeed, for Example \ref{exam:1}, if \red{$M_p \geq pL_p^{F_1}$ and $F_2$ is convex}, then the surrogate function $s$ is convex and hence the problem \eqref{eq:sub_prbm} is convex (see Theorem 1 \cite{Nes:20}), while for Example \eqref{exam:2}, the surrogate is convex if $M_p\geq 0$. \red{Hence, in the convex case we assume that $x_{k+1}$ is the global optimum of the subproblem \eqref{eq:sub_prbm}}. However, in the nonconvex case, we cannot \red{guarantee} the convexity of the subproblem. {In this case, we either assume that we can compute a stationary point of the subproblem \eqref{eq:sub_prbm} if $g$ is the identity function  or we can compute an inexact solution as defined in \eqref{eq:inx} if $g$ is a general function.} Note that our algorithmic framework is quite general and yields an array of algorithms, each of which is associated with the specific properties  of $F$ and the corresponding surrogate.  For example, if $F$ is a sum between a smooth term and a nonsmooth one we can use a surrogate as in Example \ref{exam:1}; if $F$ is fully nonsmooth we can use a surrogate as in Example \ref{exam:2}.  This is the first time such an analysis is performed, and most of our variants of GCHO were not explicitly considered in the literature before (especially in the nonconvex settings). Note that in both Examples \ref{exam:1} and \ref{exam:2}, $x_{k+1}$ can be computed  \textit{inexactly}, as detailed in the next sections.


\subsection{Nonconvex convergence analysis}\label{subsec1}
In this section we consider that each $F_{i}$, with $i=1:m$, and $h$ are nonconvex functions (possible nonsmooth). Then, problem \eqref{eq:optpb} becomes a pure nonconvex optimization problem. Now we are ready to analyze the convergence behavior of GCHO algorithm under these general settings.
In the sequel, we assume that $g(-R_{p}^{e})<0$. Note that since the vector  $R_{p}^{e}>0$, then for all the optimization problems considered in  Examples \ref{expl:3},  \ref{expl:4} and \ref{expl:5} this assumption holds provided that $M_p$ is large enough.

\vspace{-0.4cm}

\begin{theorem}\label{Th:ncx}
Let $F$, $g$ and $h$ satisfy Assumption \ref{ass:fun} and additionally each $F_{i}$ admits a $p$ higher-order surrogate $s_{i}$ as in Definition \ref{def:surg} with the constants $L^{e}_{p}(i)$ and  $R_{p}^{e}(i)$, for $i=1:m$. Let $\left(x_{k}\right)_{k\geq 0}$ be the sequence generated by Algorithm GCHO, $R_{p}^{e}=\big(R_{p}^{e}(1),\cdots,R_{p}^{e}(m)\big)$ and $L^{e}_{p}=\big(L^{e}_{p}(1),\cdots,L^{e}_{p}(m)\big)$. Then, the sequence $\left( f(x_{k})\right)_{k\geq 0}$ is nonincreasing and satisfies the following descent relation:
\begin{align}\label{eq:lem1}
f(x_{k+1})\leq f(x_{k})+\frac{g(-R_{p}^{e})}{(p+1)!}\norm{x_{k+1}-x_{k}}^{p+1}\qquad  \forall k\geq 0.
\end{align}
\end{theorem}

\vspace{-0.6cm}

\begin{proof}
Denote $e(x_{k+1};\!\! x_{k})=\big(e_{1}(x_{k+1};\!\!x_{k}),\cdots,e_{m}(x_{k+1};\!\!x_{k})\big)$. Then, from the definition of the error function $e$ and \eqref{sur:p4}, we have:

\vspace{-0.3cm}

\begin{align*} 
 \frac{\red{R_p^{e}}}{(p+1)!}\norm{x_{k+1} - x_k}^{p+1}\leq e(x_{k+1};x_k) = s(x_{k+1};x_k) - F(x_{k+1}).
\end{align*}

\noindent This implies that:

\vspace{-0.5cm}

\begin{align*}
F(x_{k+1}) \leq s(x_{k+1};x_k) - \frac{\red{R_p^{e}}}{(p+1)!}\norm{x_{k+1} - x_k}^{p+1}.
\end{align*}

\vspace{-0.2cm}

\noindent Since $g$ is nondecreasing, we get:
\begin{align*}
g(F(x_{k+1})) &\leq g\left(s(x_{k+1};x_k) - \frac{\red{R_p^{e}}}{(p+1)!}\norm{x_{k+1} - x_k}^{p+1}\right)\\ &\stackrel{\eqref{Lemma:1}}{\leq} g\big(s(x_{k+1};x_k)\big) +\frac{g(-\red{R_p^{e}})}{(p+1)!}\norm{x_{k+1} - x_k}^{p+1}.
\end{align*}

\vspace{-0.2cm}

\noindent Finally, we obtain that:

\vspace{-0.5cm}

\begin{align*}
f(x_{k+1})&\leq g\big(s(x_{k+1};x_k)) + h(x_{k+1}) + \frac{g(-\red{R_p^{e}})}{(p+1)!}\norm{x_{k+1} - x_k}^{p+1}\\
&\stackrel{\eqref{eq:loc}}{\leq} f(x_k) + \frac{g(-\red{R_p^{e}})}{(p+1)!}\norm{x_{k+1} - x_k}^{p+1},
\end{align*}
%
which yields our statement.
\end{proof}

\noindent Summing \eqref{eq:lem1} from $j=0$ to $k$, we get:
\begin{align*}
\sum_{j=0}^{k}-\frac{g(-R_{p}^{e})}{(p+1)!}\norm{x_{j+1}-x_{j}}^{p+1}&\leq \sum_{j=0}^{k} f(x_{j})-f(x_{j+1})\\&= f(x_{0})-f(x_{k+1})\leq f(x_{0})-f^{*}. 
\end{align*}

\vspace{-0.4cm}

\noindent Taking the limit as $k\to+\infty $, we obtain:

\vspace{-0.6cm}

\begin{align}
\sum\limits_{k=0}^{+\infty}\norm{x_{k}-x_{k+1}}^{p+1}< +\infty.
\end{align}

\vspace{-0.3cm}

\noindent Hence $\lim_{k\mapsto+\infty}\norm{x_{k}-x_{k+1}}=0$.
In our convergence analysis, we also consider the following additional assumption which requires the existence of some auxiliary sequence that must be closed to the sequence generated by GCHO algorithm and some first-order relation holds:

\vspace{-0.4cm}

\begin{assumption}\label{cond:2}
Given the sequence $\big(x_{k}\big)_{k\geq 0}$ generated by GCHO algorithm, there exist two constants  $L^{1}_{p}, L^{2}_{p}>0$ and a sequence $\left(y_{k}\right)_{k\geq 0}$ such that:
\begin{align}\label{stp:yk}
\norm{y_{k+1}\!-\!x_{k}}\!\leq \!L^{1}_{p}\norm{x_{k+1}\!-\!x_{k}} \;\text{and}\;\;
S_{f}(y_{k+1})\!\leq\! L^{2}_{p}\norm{y_{k+1}\!-\!x_{k}}^{p}  \;\;\;   \forall k\geq 0.
\end{align}
\end{assumption}

\vspace{-0.4cm}

\noindent In the next section, we provide concrete examples for the sequence $(y_{k})_{k\geq 0}$  satisfying \red{Assumption \ref{cond:2}}, and the corresponding expressions for $L_{p}^1$ and $L_{p}^2$.

  
\subsection{Approaching the set of stationary points}
\label{sec:artseq}
Before continuing with the convergence analysis of GCHO algorithm, let us analyze the relation between $\norm{x_{k+1}-x_{k}}^{p}$ and $S_{f}(x_{k+1})$ and also give examples when Assumption \ref{cond:2} is satisfied. For simplicity,  consider the following simple composite minimization problem:

\vspace{-0.6cm}

\begin{align*}
\min\limits_{x} f(x):=F(x)+h(x),
\end{align*}

\vspace{-0.3cm}

\noindent where $F$ is $p$ times differentiable function, having the $p$ derivative $L_{p}^{F}$-Lipschitz and $h$ is proper lower semicontinuous function. In this case $g$ is the identity function and we can take as a surrogate $s(y;\!x)=T_{p}^{F}(y;\!x) + \frac{M_{p}}{(p+1)!}\norm{x-y}^{p+1} + h(y)$, with the positive constant $M_{p}$ satisfying $M_{p} > L_{p}^{F}$ and $g(-R_p^e)<0$. The following lemma gives an example when Assumption \ref{cond:2} holds.

\vspace{-0.3cm}

\begin{lemma}\label{lem:1}
Assume $g$ is the identity function, $F$ has the $p$ derivative Lipschitz and $x_{k+1}$ is a stationary point of the following subproblem:
\begin{equation}\label{alg:MHO}
x_{k+1}\in\argmin\limits_{x} T_{p}^{F}(x;\!\!x_{k})+\dfrac{M_{p}}{(p+1)!}\norm{x-x_{k}}^{p+1} + h(x).
\end{equation}
Then, Assumption \ref{cond:2} holds with $y_{k+1}=x_{k+1}$, $L^{1}_{p}=1$ and $L^{2}_{p}=\frac{M_{p} + L_{p}^{F}}{p!}$.
\end{lemma}

\vspace{-0.6cm}

\begin{proof}
Since $x_{k+1}$ is a stationary point of  subproblem \eqref{alg:MHO}, using \eqref{eq:ch_rl}, we get:
\begin{align*}
\textcolor{black}{\dfrac{M_{p}}{p!}\norm{x_{k+1}-x_{k}}^{p-1}B (x_{k}-x_{k+1})- \nabla T_{p}^{F}(x_{k+1};\!x_{k})\in\partial h(x_{k+1}),}
\end{align*}
or equivalently
\begin{align*}
&\dfrac{M_{p}}{p!}\norm{x_{k+1}-x_{k}}^{p-1}B(x_{k}-x_{k+1})+\Big(\nabla F(x_{k+1})-\nabla T_{p}^{F}(x_{k+1};\!x_{k})\Big)\\
&\in \nabla F(x_{k+1})+\partial h(x_{k+1})=\partial f(x_{k+1}).
\end{align*}
Taking into account that $F$ is $p$-smooth, we further get:
\begin{align}\label{eq:eq0}
S_{f}(x_{k+1})&\leq \dfrac{M_{p}}{p!}\norm{x_{k+1}-x_{k}}^{p} + \norm{\nabla F(x_{k+1})-\nabla T_{p}^{F}(x_{k+1},x_{k})}_{\ast}\\
&\stackrel{\eqref{eq:TayAppG2}}{\leq} \dfrac{M_{p}+L_{p}^{F}}{p!}\norm{x_{k+1}-x_{k}}^{p}.\nonumber   
\end{align}  

\vspace{-0.3cm}

\noindent Hence, Assumption \ref{cond:2} \red{holds with $y_{k+1} = x_{k+1}$, $L^{1}_{p}=1$ and $L^{2}_{p}=\frac{M_{p} + L_{p}^{F}}{p!}$ }.
\end{proof}


\noindent The algorithm GCHO which generates a sequence $(x_k)_{k \geq 0}$  satisfying the descent \eqref{eq:loc} and the stationary condition \eqref{alg:MHO} has been also considered e.g., in the recent papers \cite{NecDan:20,Nes:20Inxt}, with $h$ assumed to be a convex function.  Here we remove this assumption on $h$.
\noindent Combining \eqref{eq:eq0} and \eqref{eq:lem1}, we further obtain:

\vspace{-0.6cm}

\begin{align*}
S_{f}(x_{k+1})^{\frac{p+1}{p}}&\leq \left(\dfrac{M_{p}+L_{p}^{F}}{p!}\right)^{\frac{p+1}{p}}\dfrac{(p+1)!}{M_{p}-L_{p}^{F}}\Big(f(x_{k})-f(x_{k+1})\Big)\\
&=C_{M_{p},L_{p}^{F}}\Big(f(x_{k})-f(x_{k+1})\Big),
\end{align*} 

\vspace{-0.4cm}

\noindent where $C_{M_{p},L_{p}^{F}}=\left(\dfrac{M_{p}+L_{p}^{F}}{p!}\right)^{\frac{p+1}{p}}\dfrac{(p+1)!}{M_{p}-L_{p}^{F}}$ . Summing the last inequality from $j=0:k-1$, and using that $f$ is bounded from bellow by $f^{*}$, we get:

\vspace{-0.6cm}

\begin{align*}
\sum^{k-1}_{j=0}S_{f}(x_{j})^{\frac{p+1}{p}}&\leq C_{M_{p},L_{p}^{F}}\Big(f(x_{0})-f(x_{k})\Big)                                      \leq C_{M_{p},L_{p}^{F}}\Big(f(x_{0})-f^{*}\Big). 
\end{align*}

\vspace{-0.3cm}

\noindent Hence:

\vspace{-1.2cm}

\begin{align*}
\min_{j=0:k-1}S_{f}(x_{j})\leq  \dfrac{\left(C_{M_{p},L_{p}^{F}}(f(x_{0})-f^{*})\right)^{\frac{p}{p+1}}}{k^{\frac{p}{p+1}}}.
\end{align*}

\vspace{-0.3cm}

\noindent Thus, we have proved convergence for the simple composite problem under slightly more general assumptions than in \cite{NecDan:20,Nes:20Inxt}, i.e., $F$ and $h$ are possibly nonconvex functions. 
\red{Finally, if we have $ \norm{x_{k+1}-x_{k}}^{p} \leq \frac{p!}{L^F_{p}+M_{p}}\epsilon$, then   from \eqref{eq:eq0} it follows that $S_f (x_{k+1})\leq \epsilon $, i.e.,  $x_{k+1}$ is nearly stationary for $f$}.
Note that in the previous Lemma \ref{lem:1}, we assume $x_{k+1}$ to be a stationary point of the following subproblem (see \eqref{alg:MHO}):

\vspace{-0.9cm}

\begin{align}\label{eq:op_cond}
x_{k+1} \in \argmin_{x} s(x;\!x_{k}).
\end{align}

\vspace{-0.3cm}

\noindent However, our stationary condition for $x_{k+1}$ can be relaxed to the \textit{following inexact} optimality criterion (\red{see also \cite{BriGar:17}}):

\vspace{-0.6cm}

\begin{align}\label{eq:inx_exmp}
\norm{g_{x_{k+1}}}\leq \theta\norm{x_{k+1}-x_{k}}^{p},
\end{align}

\vspace{-0.2cm}

\noindent where $g_{x_{k+1}}\in\partial s(x_{k+1};\!x_{k}) $ and $\theta>0$. For simplicity of the exposition, in our convergence analysis below for this particular case (i.e., $g$ identity function) we assume however that $x_{k+1}$ satisfies the exact stationary condition \eqref{eq:op_cond}, although our results can be extended to the \textit{inexact} stationary condition from above. The situation is dramatically different for the general composite problem \eqref{eq:optpb}. When $g$ is nonsmooth, the distance $\text{dist}\big(0,\partial f(x_{k+1})\big)$ will typically not even tend to zero in the limit, although we have seen that $\norm{x_{k+1}-x_{k}}^{p}$ converges to zero. Indeed, consider the  minimization of the following function: 

\vspace{-0.3cm}

$$f(x)=\max\big(x^{2}-1,1-x^{2}\big).$$

\noindent For $p=1$, we have $L_{1}^{F}(1)=L_{1}^{F}(2)=2$. Taking $x_{0}>1$ and $M_{1}=M_{2}=4$, GCHO algorithm becomes:

\vspace{-0.7cm}

\begin{align*}
x_{k+1}=\argmin\limits_{x} Q(x,x_{k})\;\Big(:=\max\big( Q_{1}(x,x_{k}) , Q_{1}(x,x_{k}) - 4xx_{k}+2x_{k}^2 + 2\big)\Big),
\end{align*}

\vspace{-0.4cm}

\noindent where $Q_{1}(x,x_{k})=2x^{2} - 2xx_{k}+ x_{k}^{2}-1 $. Let us prove by induction that $x_{k}> 1$ for all $k\geq 0$. Assume that $x_{k}>1$ for some $k\geq 0$.  
We notice that the polynomials $Q_{2}(x,x_{k}):=Q_{1}(x,x_{k}) - 4xx_{k}+2x_{k}^2 + 2$ and $Q_{1}(x,x_{k})$ are 2-strongly convex functions and they intersect in a unique point $\bar{x}=\frac{x_{k}^2+1}{2x_{k}}$. Also, the minimum of $Q_{2}$ is $\bar x_{2}=\frac{3}{2} x_{k}$ and the minimum of $Q_{1}$ is $\bar x_{1}:=\frac{1}{2}x_{k}$, satisfying $\bar x_{1}\leq \bar{x}\leq  \bar x_{2}$. Let us prove that $x_{k+1}=\bar{x}$. Indeed, if $x\leq \bar{x}$, then $Q(x,x_{k}) = Q_{2}(x,x_{k})$ and it is nonincreasing on $(-\infty,\bar{x}]$. Hence, $Q(x,x_{k})\geq Q(\bar{x},x_{k})$ for all $x\leq \bar{x}$. Further, if $x\geq \bar{x}$, then $Q(x,x_{k}) = Q_{1}(x,x_{k})$ and it is nondecreasing on $ [\bar{x},+\infty)$. In conclusion, $Q(x,x_{k})\geq Q(\bar{x},x_{k})$ for all $x\leq \bar{x}$. Finally, we have that: $Q(x,x_{k})\geq Q(\bar{x},x_{k})$ for all $x\in\mathbb{R}$. Since $x_{k}>1$, we also get that $x_{k+1}=\frac{x_{k}^2+1}{2x_{k}} > 1$. Since $x_{k}>1$, then $\partial f(x_{k})=2x_{k} > 2 $ and $S_{f}(x_{k})\geq 2 > 0$. Moreover, $x_{k+1} < x_{k}$ and bounded below by 1, thus $(x_{k})_{k\geq 0}$ is convergent and its limit is $1$. Indeed, assume that $x_{k}\to\hat{x}$ as $k \to \infty$. Then, we get $ \hat{x} = \frac{\hat{x}^{2} + 1}{2\hat{x}}$ and thus $\hat{x} = 1$ (recall that $\hat{x}\geq 1$). Consequently, $\norm{x_{k+1} - x_{k}}$ also converges to 0. Therefore, we must look elsewhere for a connection between $S_{f}(\cdot)$ and $\norm{x_{k+1} - x_{k}}^{p}$.

\medskip 

\noindent  Let us now consider the following subproblem:

\vspace{-0.6cm}

\begin{align}\label{eq:prox_ho}
 \mathcal{P}(x_{k})=\argmin\limits_{y} \mathcal{M}_{p}(y,x):= f(y)+\dfrac{\mu_{p}}{(p+1)!}\norm{y-x_k}^{p+1}, 
\end{align}   

\vspace{-0.3cm}

\noindent where $\mu_{p}>g(L^{e}_{p})$. Since $f$ is assumed bounded from bellow, then for any fixed  $x$, the function $y\mapsto \mathcal{M}_{p}(y,x)$ is coercive and hence the optimal value $\mathcal{M}_{p}^* =\inf\limits_{y}\mathcal{M}_{p}(y,x)$ is finite. Then,  the subproblem \eqref{eq:prox_ho} is equivalent to:

\vspace{-0.6cm}

\begin{align*}
\inf\limits_{y\in \mathcal{B}_k} f(y)+\dfrac{\mu_{p}}{(p+1)!}\norm{y-x_{k}}^{p+1}, 
\end{align*}

\vspace{-0.3cm}

\noindent for some   compact set $\mathcal{B}_k$.  Since $\mathcal{M}_{p}$ is proper lower semicontinuous function  in the first argument and $\mathcal{B}_k$ is compact set, then from Weierstrass theorem we have that the infimum $\mathcal{M}_{p}^*$ is attained, i.e., there exists  $\bar{y}_{k+1} \in \mathcal{P}(x_{k})$ such that  $\mathcal{M}_{p}(\bar{y}_{k+1},x_{k}) = \mathcal{M}_{p}^*$.  Since the level sets of $y \mapsto \mathcal{M}_{p}(x,y)$ are compact, then $ \mathcal{P}(x_{k})$ is nonempty and compact and one can consider the following point: 

\vspace{-0.8cm}

\begin{align}\label{stp:prx}
y_{k+1} =\argmin_{y\in \mathcal{P}(x_{k}) } \norm{y-x_{k}}.
\end{align}

\vspace{-0.3cm}
\noindent  Let us assume that $F_{i}$ admits a higher-order surrogate as in Definition \ref{def:surg}, where the error functions $e_{i}$ are $p$ smooth with Lipschitz constants $L_{p}^{e}(i)$ for all $i=1:m$. Denote $L_{p}^{e}=\big(L_{p}^{e}(1),\cdots,L_{p}^{e}(m)\big)$ and define the following positive constant $C_{L^{e}_{p}}^{\mu_{p}}= \dfrac{\mu_{p}}{\mu_{p} - g(L^{e}_{p})}$ \big(recall that $\mu_{p}$ is chosen such that $\mu_{p}>g(L^{e}_{p})$\big). {Next lemma shows that Assumption \ref{cond:2} holds provided that we compute $x_{k+1}$ as an approximate local solution of subproblem \eqref{eq:sub_prbm} (hence, $x_{k+1}$ doesn't need to be global optimum)} and $y_{k+1}$ as in \eqref{stp:prx}. 

\vspace{-0.5cm}

\begin{lemma}\label{res:lem}
Let the assumptions of Theorem \ref{Th:ncx} hold and additionally there exists $\delta>0$ such that $x_{k+1}$ satisfies the following inexact optimality condition:
\begin{align}\label{eq:inx}
g\big(s(x_{k+1};\!\!x_k)\big) + h(x_{k+1}) - \!\!\!\min_{\red{ {x: \; \|x - x_k\|\leq D_k }}}  & \left(g\big(s(x;\!\!x_k)\big) + h(x) \right) \\ 
& \qquad \leq \delta\norm{x_{k+1}-x_k}^{p+1}, \nonumber 
\end{align}

\vspace{-0.3cm}

 \noindent \red{where $D_k:= \left(\frac{(p+1)!}{\mu_p}(f(x_k) - f^*)\right)^{\frac{1}{p+1}}$}. Then,  Assumption \ref{cond:2} holds with $y_{k+1}$ given in \eqref{stp:prx}, $L^{1}_{p}=\left(C_{L^{e}_{p}}^{\mu_{p}} + \frac{\red{\delta(p+1)!}}{\mu_p - g(L^{e}_p)}\right)^{1/(p+1)}$ and $L^{2}_{p}=\frac{\mu_{p}}{p!}$.
\end{lemma}

\vspace{-0.5cm}

\begin{proof}
From the definition of $y_{k+1}$ in \eqref{stp:prx}, we have:
\begin{align}\label{eq:011}
f(y_{k+1})+\frac{\mu_{p}}{(p+1)!}\norm{y_{k+1}-x_{k}}^{p+1} &= \min\limits_{y} f(y) + \frac{\mu_{p}}{(p+1)!}\norm{y-x_{k}}^{p+1}\\ \nonumber
        &\leq f(x_{k+1}) + \frac{\mu_{p}}{(p+1)!}\norm{x_{k+1}-x_{k}}^{p+1}.
\end{align}
\red{Further, taking $y=x_k$ in \eqref{eq:011} we also have:
\begin{align*}
    f(y_{k+1}) + \frac{\mu_{p}}{(p+1)!}\norm{y_{k+1}-x_{k}}^{p+1} \leq f(x_k),
\end{align*}
which implies that:
\vspace{-0.3cm}
\begin{align}\label{eq:bdy}
\|y_{k+1} - x_k\|\leq  \left(\frac{(p+1)!}{\mu_p}(f(x_k) - f^*)\right)^{\frac{1}{p+1}} = D_k.    
\end{align}}
Note that since the error functions $e_{i}$'s have the $p$ derivative Lipschitz with constants $L_{p}^{e}(i)$'s, then using \eqref{eq:TayAppBound}, we get:
\begin{align*}
    \vert e_{i}(y;\!\!x_{k}) - T_p^{e_{i}}(y;\!\!x_{k}) \vert \leq  \frac{L_p^{e}(i)}{(p+1)!} \norm{y-x_{k}}^{p+1} \quad \forall i=1:m,\quad \forall y\in\text{dom}\,e_i.
\end{align*}
\red{From \eqref{sur:p3}}, the Taylor approximations of $e_{i}$'s of order $p$ at $x_{k}$, $T_p^{e}(y;\!\!x_{k})$, are zero. Hence we get:

\vspace{-0.4cm}

\begin{align}\label{eq:qq1}
    \vert s_{i}(y;\!\!x_{k}) - F_{i}(y) \vert=\vert e_{i}(y;\!\!x_{k}) \vert \leq  \frac{L_p^{e}(i)}{(p+1)!} \norm{y-x_{k}}^{p+1} \quad \forall i=1:m.
\end{align}
Further, since $F(x_{k+1})\leq s(x_{k+1};\!\!x_{k})$ (\red{see \eqref{sur:p4}}) and $g$ is a nondecreasing function,  we have:
\begin{align*}
f(x_{k+1}) &\leq g \Big(s(x_{k+1};\!\!x_{k})\Big) + h(x_{k+1})\\ 
           &\stackrel{\eqref{eq:inx}}{\leq} \min_{\red{ y:\;\|y - x_k\|\leq D_k }} g \Big(s(y;\!\!x_{k})\Big) + h(y) + \delta\norm{x_{k+1} - x_k}^{p+1}\\
           &\stackrel{\eqref{eq:qq1}}{\leq} \min_{\red{y:\;\|y - x_k\|\leq D_k}}  g\left(F(y) + \frac{L_p^{e}}{{p+1}!}\norm{y \!-\! x_k}^{p+1}\right) \!+\! h(y) \!+\! \delta\norm{x_{k+\!1} \!-\! x_k}^{p+1}\\
           &\stackrel{\eqref{Lemma:1}}{\leq} \min_{\red{y:\;\|y - x_k\|\leq D_k}} f(y) + \frac{g(L^{e}_p)}{(p+1)!}\norm{y - x_k}^{p+1} + \delta\norm{x_{k+1} - x_k}^{p+1}\\ 
           &\stackrel{\eqref{eq:bdy}}{\leq} f(y_{k+1}) + \dfrac{g(L^{e}_{p})}{(p+1)!}\norm{y_{k+1}-x_{k}}^{p+1} + \delta\norm{x_{k+1} - x_k}^{p+1}.         
\end{align*}
Then, combining the last inequality with \eqref{eq:011}, we get:
\begin{align*}
\norm{y_{k+1}-x_{k}}^{p+1} \leq \dfrac{\mu_{p} + \red{\delta(p+1)!}}{\mu_{p} - g(L^{e}_{p})} \norm{x_{k+1}-x_{k}}^{p+1},
\end{align*}
which is the first statement of Assumption \ref{cond:2}. Further, using \eqref{eq:ch_rl} and optimality conditions for $y_{k+1}$, we obtain:
\begin{align*}
0 \in \partial f(y_{k+1}) + \frac{\mu_{p}}{p!}\norm{y_{k+1}-x_{k}}^{p-1} B(y_{k+1}-x_{k}).
\end{align*}

\vspace{-0.3cm}

\noindent It follows that:

\vspace{-0.5cm}
 
\begin{align*}
S_{f}(y_{k+1})\leq \frac{\mu_{p}}{p!}\norm{y_{k+1}-x_{k}}^{p}.
\end{align*}

\vspace{-0.3cm}

\noindent Hence, \red{Assumption \ref{cond:2} holds with $y_{k+1}$ given in \eqref{stp:prx}, $L^{1}_{p}=\left(C_{L^{e}_{p}}^{\mu_{p}} + \frac{\delta(p+1)!}{\mu_p - g(L^{e}_p)}\right)^{1/(p+1)}$ and $L^{2}_{p}=\frac{\mu_{p}}{p!}$.}
\end{proof}

 \noindent Finally, we provide a third (practical) example satisfying Assumption \ref{cond:2} when $p=2$, \blue{$h(\cdot) = 0$} and $g(\cdot) = \max(\cdot)$ function.
 \begin{lemma}\label{Lem:3}
Let the assumptions of Theorem \ref{Th:ncx} hold and additionally assume that $p=2$, $g(\cdot) = \max(\cdot)$ and the surrogate function $s(\cdot;\cdot)$ is given in Example \ref{exam:1} with $F_2 = 0$. Then, the global solution  of the subproblem \eqref{eq:sub_prbm} with $h=0$, denoted $x_{k+1}$, can be computed efficiently and consequently Assumption \ref{cond:2} holds with $y_{k+1}$ given in \eqref{stp:prx}, $L^{1}_{p}=\left(C_{L^{e}_{p}}^{\mu_{p}}\right)^{1/3}$ and $L^{2}_{p}=\frac{\mu_{p}}{2}$.
 \end{lemma}
 \vspace*{-0.5cm}
 \begin{proof}
     See appendix.
 \end{proof}
 \noindent 
\blue{From} the proof of Lemma \ref{Lem:3} one can see that  \textit{the global minimum} of  subproblem \eqref{eq:sub_prbm} can be computed as:
\begin{align*}
x_{k+1} =x_k - H_k(u,w)^{-1}g_k(u),
\end{align*}
where $H_k(u,w) = \sum_{i=1}^{m} u_i \nabla^2 F_i(x_k) + \frac{w}{2} I$, $g_k(u) = \sum_{i=1}^{m}u_i \nabla F_i(x_k)$ and $l_k(u) = \sum_{i=1}^{m} u_i F_i(x_k)$, with $(u,w)$ the solution of the following dual problem:
\begin{align}\label{eq:d-sbp}
&\max_{(u,w)\in D}\; l_k(u) -\frac{1}{2} \left\langle H_k(u,w)^{-1} g_k(u) , g_k(u)\right\rangle - \frac{1}{12(\sum_{i=1}^{m}u_i M_i)^2} w^3,
\end{align}
with $D = \left\{ (u,w)\in\Delta_m\times\mathbb{R}_+:\; \; \text{s.t.}\; H_k(u,w) \succ 0  \right\}$, 
i.e., a maximization of a concave function over a convex set $D$. Hence, if $m$ is not too large, this convex dual problem can be solved efficiently by interior point methods \cite{NesNem:94}. In conclusion, GCHO algorithm can be implementable for $p=2$ even for nonconvex problems, since we can effectively compute the global minimum $x_{k+1}$ of subproblem \eqref{eq:sub_prbm} using the powerful tools from convex optimization.


\noindent Define the following constant: $D_{R_{p}^{e},L^{1,2}_{p}}=\frac{\left(L_{p}^{1}\left(L_{p}^{2}\right)^{p}\right)^{\frac{p+1}{p}}(p+1)!}{-g(-R_{p}^{e})}$. Then, we  derive the following convergence result for GCHO algorithm in the nonconvex case.

\vspace{-0.5cm}

\begin{theorem}\label{th:MHO}
Let the assumptions of Theorem \ref{Th:ncx} hold. Additionally,  \blue{Assumption} \ref{cond:2} holds. Then, for the sequence $\left( x_{k}\right)_{k\geq 0}$ generated by Algorithm GCHO we have the following sublinear convergence rate:

\vspace{-0.5cm}

\begin{align*}
\min\limits_{j=0:k-1} S_{f}(y_{j})\leq  \frac{\left(D_{R_{p}^{e},L^{1,2}_{p}}(f(x_{0})-f^{*})\right)^{\frac{p}{p+1}}}{k^{\frac{p}{p+1}}}.
\end{align*}\end{theorem}

\vspace{-0.6cm}

\begin{proof}
From \blue{Assumption} \ref{cond:2}, we have:
\begin{align*}
S_{f}(y_{k+1})&\leq L_{p}^{2}\norm{y_{k+1}-x_{k}}^{p}\leq L_{p}^{2}\left(L_{p}^{1}\right)^{p}\norm{x_{k+1}-x_{k}}^{p}.
\end{align*}
Using the descent \eqref{eq:lem1}, we get:
\begin{align*}
S_{f}(y_{k+1})^{\frac{p+1}{p}}\leq \frac{\left(L_{p}^{2}\left(L_{p}^{1}\right)^{p}\right)^{\frac{p+1}{p}}(p+1)!}{-g(-R_{p}^{e})}\left( f(x_{k})-f(x_{k+1})\right).
\end{align*}    
Summing the last inequality from $j=0:k-1$ and taking the minimum, we get:
\begin{align*}
\min\limits_{j=0:k-1} S_{f}(y_{j})\leq  \frac{\left(D_{R_{p}^{e},L^{1,2}_{p}}(f(x_{0})-f^{*})\right)^{\frac{p}{p+1}}}{k^{\frac{p}{p+1}}},
\end{align*} 
 which proves the statement of the theorem.
\end{proof}
\noindent Theorem \ref{th:MHO} \blue{requires} that $x_{k+1}$ satisfies the descent \eqref{eq:loc} and Assumption \ref{cond:2}. However Assumption \ref{cond:2}, according to Lemmas \ref{lem:1} and  \ref{res:lem}, holds if $x_{k+1}$ is an (inexact) stationary point or an  inexact solution of the subproblem \eqref{eq:sub_prbm},  respectively.

\vspace{-0.6cm}

\begin{remark}
To this end, Assumption \ref{cond:2} requires an auxiliary  sequence $y_{k+1}$ satisfying:
\begin{equation}\label{eq:cve}
 \left\{
    \def\arraystretch{1.8}
     \begin{array}{ll}
        \norm{y_{k+1}-x_{k}}\leq L_{p}^{1}\norm{x_{k+1}-x_{k}} \\
         S_{f}(y_{k+1})\leq L_{p}^{2} \norm{x_{k+1}-x_{k}}^{p}.   
     \end{array}
     \right. 
\end{equation} 
If $\norm{x_{k+1}-x_{k}}$ is small, the point $x_{k}$ is near $y_{k+1}$, which is nearly stationary for $f$ (recall that $\norm{x_{k+1}-x_{k}}$ converges to $0$). Hence, we do not have approximate stationarity for the original sequence $x_k$ but for the auxiliary sequence $y_k$, which is close to the original sequence.   Note  that in practice, $y_{k+1} $ does not need to be computed. The purpose of $y_{k+1}$ is to certify that $x_{k}$ is approximately stationary in the sense of \red{\eqref{eq:cve}}.  For $p=1$ a similar conclusion was derived in \cite{DruPaq:19}.  For a better understanding of the behavior of the sequence $y_{k+1}$, let us come back to our example  $f(x)=\max\big(x^{2}-1,1-x^{2}\big)$  and  $p=1$. Recall that we have proved  $x_{k}>1$ and choosing $\mu_{p}=4$,  then $y_{k+1}$ is the solution of the following subproblem:

\vspace{-0.3cm}

\begin{align*}
y_{k+1} &= \argmin\limits_{y} \max\big(y^{2}-1,1-y^{2}\big) + 2(y-x_{k})^2.
\end{align*}
\vspace{-0.3cm}

\noindent Then, it follows immediately that:
\begin{equation*}
 y_{k+1} = \left\{
    \def\arraystretch{1.8}
     \begin{array}{ll}
        \frac{2}{3}x_{k}\quad & \text{if} \;\;x_k > \frac{3}{2} \\
         1\quad & \text{if} \;\; 1\leq x_{k} \leq\frac{3}{2}.   
     \end{array}
     \right. 
\end{equation*}
Since we have already proved that $x_k \to 1$, we conclude that $\vert y_{k+1} - x_{k} \vert \to 0$ and consequently  $\text{dist}(0,\red{\partial f(y_{k+1})})\to 0$ for $k \to \infty$, as predicted by our theory. 
\end{remark}


\subsection{Better rates for GCHO under KL}
In this section, we show that improved rates can be derived for GCHO algorithm if the objective function satisfies the KL property.  This is the first time when such convergence  analysis is derived for the GCHO algorithm on the composite problem \eqref{eq:optpb}. We believe that this lack of analysis comes from the fact that, \red{when $g$ is nonsmooth and different from the identity function}, one can't bound directly the distance $S_{f}(x_{k+1})$ by $\norm{x_{k+1} - x_{k}}$. However, using the newly introduced   (artificial) point $y_{k+1}$, we can now overcome this difficulty. Let us recall the definition of a function satisfying the \textit{Kurdyka-Lojasiewicz (KL)} property (see \cite{BolDan:07} for more details).

\vspace{-0.3cm}

\begin{definition}
\label{def:kl}
\noindent A proper lower semicontinuous  function $f: \mathbb{E}\rightarrow \bar{\mathbb{R}}$ satisfies  \textit{Kurdyka-Lojasiewicz (KL)} property on the compact set $\Omega \subseteq \text{dom} \; f$ on which $f$ takes a constant value $\red{f_*}$ if there exist $\delta, \epsilon >0$ such that   one has:
\begin{equation*}
\kappa' (f(x) - \red{f_*}) \cdot  S_{f}(x)  \geq 1  \quad   \forall x\!:  \text{dist}(x, \Omega) \leq \delta, \;  \red{f_*} < f(x) < \red{f_*} + \epsilon,  
\end{equation*}
where $\kappa: [0,\epsilon] \to \mathbb{R}$ is  concave differentiable function satisfying $\kappa(0) = 0$ and $\kappa'>0$.
\end{definition}    

\vspace{-0.5cm}

\noindent When  $ \kappa$ takes the form  $\kappa (t) = \sigma_q^{\frac{1}{q}} \frac{q}{q-1} t^{\frac{q-1}{q}}$, with $q >1$ and $\sigma_q>0$ (which is our interest here), the KL property establishes the following local geometry of the nonconvex function $f$ around a compact set~$\Omega$:

\vspace{-0.6cm}

\begin{equation}\label{eq:kl}
f(x) - \red{f_*}  \leq \sigma_q  S_{f}(x)^q \quad   \forall x\!: \;  \text{dist}(x, \Omega) \leq \delta, \; \red{f_*} < f(x) < \red{f_*} + \epsilon.  
\end{equation}

\vspace{-0.3cm}

\noindent  Note that the relevant aspect of the KL property is when $\Omega$ is a subset of stationary points for $f$, i.e.  $\Omega \subseteq \{x: 0 \in \partial f(x) \}$, since it is easy to establish the KL property when $\Omega$ is not related to stationary points. The KL property holds for a large class of functions including semi-algebraic functions (e.g., real polynomial functions), vector or matrix (semi)norms (e.g., $\|\cdot\|_p$ with $p \geq 0$ rational number), logarithm functions,  exponential functions and  uniformly convex functions,  see \cite{BolDan:07} for a comprehensive list. \red{ In particular,  the max (sup) of semi-algebraic functions is a semi-algebraic function, see \cite{BolSab:14} (Example 2).} Let us show that if $(x_{k})_{k\geq 0}$ is bounded, then also $( y_{k})_{k\geq 0}$ is bounded and they have the same limit points. 


\vspace{-0.3cm}

\begin{lemma}\label{rmq1}
Let $\left(x_{k}\right)_{k\geq 0}$ generated by Algorithm GCHO be bounded and $(y_{k})_{k\geq 0}$ \blue{satisfies} Assumption \ref{cond:2}. Then,  $(y_{k})_{k\geq 0}$ is bounded and the set of limit points of the sequence $\left(y_{k}\right)_{k\geq 0}$ coincides with the set of limit points of $\left( x_{k}\right)_{k\geq 0}$.
\end{lemma}

\vspace{-0.6cm}

\begin{proof}
Indeed, since $\left(x_{k}\right)_{k\geq 0}$ is bounded, then it has limit points.  Let $x_{*}$ be a limit point of the sequence $\left(x_{k}\right)_{k\geq0}$. Then, there exists a subsequence $( x_{k_{t}})_{t\geq0}$ such that $x_{k_{t}}\to x_{*}$ for $t\to\infty$. We have:
\begin{align}\label{eq:la}
\norm{y_{k_{t}}-x_{k_{t}}}&\leq\norm{y_{k_{t}}-x_{k_{t}-1}}+\norm{x_{k_{t}}-x_{k_{t}-1}}\\
     &\stackrel{\eqref{stp:yk}}{\leq} \left(L^{1}_{p} + 1\right)\norm{x_{k_{t}}-x_{k_{t}-1}}\qquad \forall k\geq 0,\nonumber
\end{align}   
Since $(x_k)_{k\geq 0}$ is bounded and $\|x_{k+1} - x_{k}\|\to 0$, then $(y_{k})_{k\geq 0}$ is also bounded. This implies that $y_{k_{t}}\to x_{*}$. Hence, $x_{*}$ is also a limit point of the sequence $\left(y_{k}\right)_{k\geq 0}$. Further, let $y_{*}$ be a limit point of the bounded sequence $\left( y_{k}\right)_{k\geq 0}$. Then, there \blue{exists} a subsequence $( y_{\bar{k}_{t}})_{t\geq 0}$ such that $y_{\bar{k}_{t}}\to y_{*}$ for $t\to\infty$. \red{Taking $t\to\infty$ in an inequality similar to \eqref{eq:la} and using $\lim_{t\to \infty} \|x_{\bar{k}_t} - x_{\bar{k}_t - 1}\| = 0 $ and boundedness of $(x_k)_{k\geq 0}$, we get that $x_{\bar{k}_{t}}\to y_{*}$}, i.e.,  $y_{*}$ is also a limit point of the sequence $\left(x_{k}\right)_{k\geq0}$.     
\end{proof}

\vspace{-0.1cm}

\noindent Note that usually for deriving convergence rates under KL condition, we need to assume that the sequence generated by the algorithm is bounded (see e.g.,  Theorem 1 in \cite{BolSab:14}). Let us denote the set of limit points of $(x_{k})_{k\geq 0}$ by:

\vspace{-0.6cm}

\begin{align*}
\Omega (x_{0})=&\lbrace \bar{x}\in \mathbb{E}: \exists \text{ an increasing sequence of integers } (k_{t})_{t\geq0},\\
& \text{ such that } x_{k_{t}}\to \bar{x} \text{ as } t\to \infty \rbrace,
\end{align*} 

\vspace{-0.4cm}

\noindent and  the set of stationary points of problem \eqref{eq:optpb} by $\text{crit}f$.

\vspace{-0.3cm}

\begin{lemma}\label{Lemma:crt_pint}
Let the assumptions of Theorem \ref{Th:ncx} hold. Additionally, assume that  $\left(x_{k}\right)_{k\geq 0}$ is bounded, $(y_{k})_{k\geq 0}$ \blue{satisfies} Assumption \ref{cond:2} and $f$ is continuous. Then, we have: $\emptyset \neq \Omega(x_{0}) \subseteq \text{crit} f $, $ \Omega(x_{0})$ is compact and connected set, and $f$ is constant on $ \Omega(x_{0})$, i.e., $f(\Omega(x_{0})) = f_*$.
\end{lemma}

\vspace{-0.6cm}

\begin{proof}
First let us show that $f(\Omega(x_{0}))$ is constant. From \eqref{eq:lem1} we have that $\left( f(x_{k})\right)_{k\geq 0}$ is monotonically decreasing and since $f$ is assumed bounded from below, it converges, let us say to $\red{f_*}>-\infty $, i.e. $f(x_{k})\to \red{f_*}$ as $k\to \infty$. On the other hand let $x_{*}$ be a limit point of the sequence $\left( x_{k}\right)_{k\geq0}$. This means that there exist a subsequence $\left( x_{k_{t}}\right)_{t\geq0}$ such that $x_{k_{t}}\to x_{*} $. Since $f$ is continuous, then  $f(x_{k_{t}})\to f(x_{*}) = \red{f_*}$.
 In conclusion, we have $f(\Omega(x_{0}))=\red{f_*}$. The closeness property of $\partial f$ implies that $S_{f}(x_{*})=0$, and thus $0\in \partial f(x_{*})$. This proves that $x_{*}$ is a stationary point of $f$ and thus $\Omega(x_{0})$ is nonempty. By observing that $\Omega(x_{0})$ can be viewed as an intersection of compact sets:
\begin{align*}
\Omega(x_{0})=\cap_{q\geq 0} \overline{\cup_{k\geq q}\lbrace x_{k}\rbrace}, 
\end{align*}
so it is also compact. This completes our proof.
\end{proof}

\noindent \red{Note that $f_*$ from Lemma \ref{Lemma:crt_pint} is usually  different from $f^* = \inf_{x\in\text{dom}f } f(x)$ defined in  Assumption \ref{ass:fun}.} In addition, let us consider the following assumption:

\vspace{-0.5cm}

\begin{assumption}\label{cond:3}
For the sequence $\big(x_{k}\big)_{k\geq 0}$ generated by GCHO algorithm, there exist positive constants $\theta_{1,p},\theta_{2,p}>0 $ such that:
\begin{align}\label{eq:ad:ass}
f(x_{k+1})\leq f(y_{k+1}) + \theta_{1,p}\norm{y_{k+1}-x_{k}}^{p+1} + \theta_{2,p}\norm{x_{k+1}-x_{k}}^{p+1}   \quad \forall k\geq 0.
\end{align}
\end{assumption}

\vspace{-1cm}

\begin{remark}\label{remark:a3}	
\noindent Note that Assumption \ref{cond:3} holds when e.g.,  $g$ is the identity function or when $(y_{k})_{k\geq 0}$ is given in \eqref{stp:prx} and $x_{k+1}$ satisfies \eqref{eq:inx} (see Lemmas  \ref{res:lem} and \ref{Lem:3}). For completeness, we provide a proof for this statement in Appendix.  
\end{remark}

\vspace{-0.3cm}

\noindent Let us also recall the following lemma, whose proof is similar to the one  in \cite{AttBol:09}(Theorem 2).  For completeness, we give the proof in Appendix.

\vspace{-0.3cm}

\begin{lemma}\label{lemcv}
Let $\theta>0$, $C_{1},C_{2}\geq 0$ and $(\lambda_{k})_{k\geq 0}$ be a nonnegative, nonincreasing sequence, satisfying the following recurrence:
\begin{align}\label{eq:recurence}
\lambda_{k+1}\leq C_{1}\left(\lambda_{k}-\lambda_{k+1}\right)^{\frac{1}{\theta}}+C_{2}\left(\lambda_{k}-\lambda_{k+1}\right).
\end{align}
If $\theta \leq 1$, then there exists an integer $k_{0}$ such that:
\begin{align*}
\lambda_{k}\leq \left(\frac{C_{1}+C_{2}}{1+C_{1}+C_{2}}\right)^{k-k_{0}}\lambda_{0} \qquad \forall k\geq k_{0}.
\end{align*}
If $\theta>1$, then there \blue{exists} $\alpha > 0$ and integer $k_{0}$ such that:
\begin{align*}
\lambda_{k}\leq \frac{\alpha}{(k-k_{0})^{\frac{1}{\theta-1}}}\qquad \forall k\geq k_{0}.
\end{align*}
\end{lemma}

\vspace{-0.5cm}

\noindent From previous lemmas, all the conditions of the KL property from Definition \ref{def:kl} are satisfied.  Then, we can derive the following convergence rates depending on the KL parameter. 

\vspace{-0.3cm}

\begin{theorem}\label{Thm:kl}
Let the assumptions of  Lemma \ref{Lemma:crt_pint} hold. Additionally, assume that $f$ \blue{satisfies} the KL property \eqref{eq:kl} on $\Omega(x_0)$ and  Assumption \ref{cond:3} is valid. Then, the following convergence rates hold for the sequence $(x_{k})_{k\geq 0}$ generated by GCHO algorithm:
\begin{enumerate}
\item[$\bullet$]If $q\geq\frac{p+1}{p}$, then $f(x_{k})$ \blue{converges} to ${f_*}$ linearly for $k$ sufficiently large.
\item[$\bullet$]If $q < \frac{p+1}{p} $, then $f(x_{k})$ \blue{converges} to ${f_*}$ at sublinear rate of order $\mathcal{O}\left(\frac{1}{k^{\frac{pq}{p+1-pq}}}\right)$ for $k$ sufficiently large.
\end{enumerate}
\end{theorem}

\vspace{-0.6cm}

\begin{proof}
Since $(x_k)_{k\geq 0}$ and $(y_k)_{k \geq 0}$ \blue{have} the same limit points, we get:
\begin{align*}
&f(x_{k+1})-{f_*}\stackrel{\eqref{eq:ad:ass}}{\leq} f(y_{k+1})-{f_*} + \theta_{1,p}\norm{y_{k+1}-x_{k}}^{p+1}+ \theta_{2,p}\norm{x_{k+1} - x_k}^{p+1}\\
                &\stackrel{\eqref{eq:kl} + \eqref{stp:yk}}\leq \sigma_{q} S_f(y_{k+1})^{q} + \left(\theta_{1,p}(L_{p}^{1})^{p+1} + \theta_{2,p}\right) \norm{x_{k+1}-x_{k}}^{p+1}\\
                &\stackrel{\eqref{stp:yk}}{\leq} \sigma_{q}\left (L_{p}^{2} (L_p^1)^p \right)^{q}\norm{x_{k+1}-x_{k}}^{qp} + \left(\theta_{1,p}(L_{p}^{1})^{p+1} + \theta_{2,p}\right)\norm{x_{k+1}-x_{k}}^{p+1}.
\end{align*}
If we define $\Delta_{k}=f(x_{k})-{f_*}$, \red{then combining the last inequality with \eqref{eq:lem1},} we get the following recurrence:
\begin{align*}
\Delta_{k+1} \leq C_{1}\left( \Delta_{k}-\Delta_{k+1}\right)^{\frac{qp}{p+1}} + C_{2}\left( \Delta_{k}-\Delta_{k+1}\right),
\end{align*}
where $ C_{1}= \sigma_{q}(L_{p}^{2}(L_p^1)^p)^{q}\left(\frac{(p+1)!}{-g(-R_{p}^{e})}\right)^{\frac{pq}{p+1}}$
and $ C_{2}= \left(\theta_{1,p}(L_{p}^{1})^{p+1} + \theta_{2,p}\right) \frac{(p+1)!}{-g(-R_{p}^{e})}.$
\noindent Using Lemma \ref{lemcv}, with $\theta=\frac{p+1}{pq}$ we get our statements.
\end{proof}

\vspace{-0.6cm}

\begin{remark}
Contrary to Theorem \ref{th:MHO}, under KL we prove in Theorem \ref{Thm:kl} that the original sequence $(x_k)_{k\geq 0}$ converge in function values. 
When the objective function $f$ is uniformly convex of order $p+1$ and $g$ not necessarily with full domain, \cite{DoiNes:21} proves linear convergence for their algorithm in function values. Our results are different, i.e., we provide convergence rates for GCHO algorithm for possibly nonconvex objective~$f$.
\end{remark}


\subsection{Convex convergence analysis}\label{sec:conv}
In this section, we assume that the objective function $f$ in \eqref{eq:optpb}  is convex. Since the problem \eqref{eq:optpb} is convex, we assume that $x_{k+1}$ is a global minimum of the subproblem \eqref{eq:sub_prbm}, which is convex provided that $M_p$ is sufficiently large (see Theorem 1 in  \cite{Nes:20}). Below, we also assume that the level sets of $f$ are bounded. Since GCHO algorithm is a descent method, this implies that there exist a constant $R_0>0$ such that $\norm{x_{k}-x^{*}}\leq R_0$ for all $k\geq 0$, where $x^*$ is an optimal solution of  \eqref{eq:optpb}. Then, we get the following sublinear rate for GCHO algorithm.

\vspace{-0.3cm}

\begin{theorem}\label{Theorem:convex}
Let $F$, $g$ and $h$ satisfy Assumption \ref{ass:fun} and additionally each $F_{i}$ admits a $p$ higher-order surrogate $s_{i}$ as in Definition \ref{def:surg} with the constants $L^{e}_{p}(i)$ and  $R_{p}^{e}(i)$, for $i=1:m$. Additionally, $f$ is a convex function and has bounded level sets. Let $\left(x_{k}\right)_{k\geq 0}$ be the sequence generated by Algorithm GCHO, $R_{p}^{e}=\big(R_{p}^{e}(1),\cdots,R_{p}^{e}(m)\big)$ and $L^{e}_{p}=\big(L^{e}_{p}(1),\cdots,L^{e}_{p}(m)\big)$. Then, we have the following convergence rate:
\begin{align*}
f(x_{k})-f(x^{*})\leq \dfrac{g(L^{e}_{p})R_0^{p+1}(p+1)^{p}}{p!k^{p}}.
\end{align*}
\end{theorem}
\begin{proof}

\vspace{-0.6cm}

\red{Since $F(x_{k+1})\leq S(x_{k+1};\!\!x_k)$ (see \eqref{sur:p4}) and $g$ is nondecreasing}, we have:
\begin{align*}
f(x_{k+1}) &\leq g\big(s(x_{k+1};\!x_{k})\big)+h(x_{k+1})\\
 &\stackrel{\eqref{eq:op_cond}}{=}\min\limits_{x} g\big(s(x;\!\!x_{k})\big)+h(x)\\
                   &\stackrel{\eqref{eq:qq1}}{\leq} \min\limits_{x} g\left( F(x)+\dfrac{L^{e}_{p}}{(p+1)!}\norm{x-x_{k}}^{p+1}\right)+h(x).
\end{align*}
Hence we get:
\begin{align*}                    
f(x_{k+1})  &\stackrel{\eqref{Lemma:1}}{\leq} \min\limits_{x} g\big(F(x)\big)+ \dfrac{g(L^{e}_p)}{(p+1)!}\norm{x-x_{k}}^{p+1} + h(x) \\  
                    &= \min\limits_{x} f(x) + \frac{g(L_p^{e})}{(p+1)!}\norm{x - x_k}^{p+1}\\                                       
                    &\leq \min\limits_{\alpha\in [0,1]}  f(x_{k})+\alpha\big[(f(x^{*})-f(x_{k})\big]+\alpha^{p+1}\dfrac{R_0^{p+1}}{(p+1)!}g\left(L^{e}_{p}\right),                     
\end{align*}
\red{where the last inequality follows from the convexity of $f$ and the boundness of the level sets of $f$.}
The minimum  in $\alpha \geq 0$ is achieved at:
\begin{align*}
\alpha^{*}=\left(\dfrac{f(x_{k})-f(x^{*})p!}{g(L^{e}_{p})R_0^{p+1}}\right)^{\frac{1}{p}}. 
\end{align*} 
We have $0\leq \alpha^{*}<1 $. Indeed, since $\big(f(x_{k})\big)_{k\geq 0}$ is decreasing, we have:
\begin{align*}
f(x_{k})\leq f(x_{1})&\leq g\big( s(x_{1};\!\!x_{0})\big)+h(x_{1})\\
                     &=\min\limits_{x} g\big(s(x;\!\!x_{0})\big)+h(x)\\
                     &\stackrel{\eqref{eq:qq1}}{\leq} \min\limits_{x} g\left( F(x)+\frac{L^{e}_{p}}{(p+1)!}\norm{x-x_{0}}^{p+1}\right)+h(x)\\ 
                     &\leq g\left( F(x^{*})+\frac{L^{e}_{p}}{(p+1)!}\norm{x^{*}-x_{0}}^{p+1}\right)+h(x^{*})\\
		             &\leq f(x^{*})+\dfrac{g(L^{e}_{p})R_0^{p+1}}{(p+1)!}.
\end{align*}

\vspace{-0.4cm}

\noindent Hence:

\vspace{-0.6cm}

\begin{align*}
0\leq\alpha^{*}&\leq \left(\dfrac{\big(f(x_{1})-f(x^{*})\big)p!}{g(L^{e}_{p})R_0^{p+1}}\right)^{\frac{1}{p}}\leq \left(\dfrac{g(L^{e}_{p})R_0^{p+1}p!}{g(L^{e}_{p})R_0^{p+1}(p+1)!}\right)^{\frac{1}{p}} \\
          &= \left(\dfrac{p!}{(p+1)!}\right)^{\frac{1}{p}}=\left(\dfrac{\red{1}}{p+1}\right)^{\frac{1}{p}}<1.
\end{align*}
Thus, we conclude:

\vspace{-0.6cm}

\begin{align*}
f(x_{k+1})&\leq f(x_{k})-\alpha^{*}\left(f(x_{k})-f(x^{*})-\frac{g(L^{e}_{p})R_0^{p+1}}{(p+1)!}(\alpha^{*})^{p}  \right) \\
             &=  f(x_{k})-\dfrac{p\alpha^*}{p+1} \big[f(x_{k})-f(x^{*})\big] .
\end{align*}
Denoting $\delta_{k}=f(x_{k})-f(x^{*})$, we get the following estimate:
\begin{align*}
\delta_{k}-\delta_{k+1}\geq C \delta_{k}^{\frac{p+1}{p}},
\end{align*}

\vspace{-0.2cm}

\noindent where $\red{C=\frac{p}{p+1}\left( \dfrac{p!}{g(L^{e}_{p})R_0^{p+1}} \right)^{\frac{1}{p}}}$\!\!. Thus, for $\mu_{k}\!=\!C^{p}\delta_{k}$ we get the following recurence:

\vspace{-0.5cm}

\begin{align*}
\mu_{k}-\mu_{k+1}\geq \mu_{k}^{\frac{p+1}{p}}.
\end{align*}
Following the same proof as in \cite{Nes:20}(Theorem 4), we get:
\begin{align*}
\dfrac{1}{\mu_{k}}\geq\left(\dfrac{1}{\mu_{1}^{\frac{1}{p}}}+\dfrac{k-1}{p}\right)^{p}.
\end{align*}
Since
$$\dfrac{1}{\mu_{1}^{\frac{1}{p}}}=\dfrac{1}{C\red{\delta_{1}}^{\frac{1}{p}}}=\frac{p+1}{p}\left(\dfrac{g(L^{e}_{p})R_0^{p+1}}{p!(f(x_{1})-f^{*})} \right) ^{\frac{1}{p}} \geq \frac{1}{p}(p+1)^{\dfrac{p+1}{p}}, $$ 
then
\begin{align*}
\delta_{k}=C^{-p}\mu_{k}&=\left(\dfrac{p+1}{p}\right)^{p}\dfrac{g(L^{e}_{p})R_0^{p+1}}{p!}\mu_{k}\\
					    &\leq \left(\dfrac{p+1}{p}\right)^{p}\dfrac{g(L^{e}_{p})R_0^{p+1}}{p!}\left(\frac{1}{p}(p+1)^{\frac{p+1}{p}}+\dfrac{k-1}{p}\right)^{-p} \\
					    &=\dfrac{g(L^{e}_{p})R_0^{p+1}}{p!}\left((p+1)^{\frac{1}{p}}+\dfrac{k-1}{p+1}\right) ^{-p}
					    \leq  \dfrac{(p+1)^{p}g(L^{e}_{p})R_0^{p+1}}{p!k^{p}}.
\end{align*}    
This proves the statement of the theorem.
\end{proof}

\noindent Note that in the convex case the convergence results from \cite{DruPaq:19,BolChe:20,DoiNes:21} assume Lipschitz continuity of the $p$ derivative of the object function $F$, which may be too restrictive. However, Theorem \ref{Theorem:convex} \blue{assumes} Lipschitz continuity of the $p$ derivative of the error function $e(\cdot)$ (note that we may have the error function $e(\cdot)$ $p$ times  differentiable and with the  $p$ derivative  Lipschitz,  while the objective function $F$ may not be even differentiable, see Examples \ref{exam:1} and \ref{exam:2}). Hence, our proof is different and more general than  \cite{DruPaq:19,BolChe:20,DoiNes:21}. Moreover, our convergence rate from the previous theorem covers   the usual convergence rates $\mathcal{O}(\frac{1}{k^{p}})$ of higher-order Taylor-based methods in the convex  unconstrained case \cite{Nes:20},  simple composite case \cite{Nes:20,Nes:20Inxt} and  composite case for  $p\geq 1$ \cite{DoiNes:21,BolChe:20}. Therefore, Theorem \ref{Theorem:convex} provides a unified convergence analysis for general composite higher-order algorithms, that covers in particular, minimax strategies for nonlinear games, min-max problems and simple composite problems, under possibly more general assumptions.


\subsection{Adaptive GCHO algorithm} 
\label{sec:ad}
In this section, we propose an adaptive variant of GCHO algorithm. Since the surrogate functions in all the examples given in this paper depend on a given constant $M$ (see Examples \ref{exam:1} and \ref{exam:2}, where $M = M_p$), below we consider the following notation $s(\cdot;\cdot): = s_{M}(\cdot;\cdot)$. Note that the convergence results from Theorems \ref{Th:ncx}, \ref{th:MHO} and \ref{Thm:kl} are derived provided that Assumption \ref{cond:2} and \ref{cond:3} and the following properties of the sequence $(x_k)_{k\geq 0}$ generated by GCHO hold:

\vspace{-0.6cm}

\begin{align}\label{eq:adp0}
 g\big(s_{M}(x_{k+1};\!x_k)\big) + h(x_{k+1}) \leq f(x_k),
\end{align}

\vspace{-1cm}

\begin{align}\label{eq:adp}
g(s_{M}(x_{k+1};\!x_k)) - g(F(x_{k+1}))\geq \frac{C_p^e}{(p+1)!}\|x_{k+1} - x_k\|^{p+1},
\end{align}

\vspace{-0.4cm}

\noindent where $C_p^e: = -g(-R_p^e)$ is a given constant depending on the choice of the surrogate $s_M(x_{k+1};x_k)$, which may be difficult to find in practice. Hence, in the following we propose an adaptive general composite higher-order algorithm, called (A-GCHO):   
\begin{center}
\noindent\fbox{%
\parbox{11.5cm}{%
\textbf{Algorithm A-GCHO}\\
Given $x_{0}\in\text{dom}\,f$, $i=0$ and $R_p,M_0 > 0$. For $k\geq 0$ \textbf{do}:\\\\
\textbf{While}\; some criterion is not satisfied
\begin{enumerate}
\item[1.]  Compute a $p$ higher-order surrogate $s_{2^{i}M_k}(\cdot;\!\!x_{k})$ of $F$ near $x_{k}$
\item[2.]Compute $x_{k+1}$  satisfying the descent \eqref{eq:adp0} with $M = 2^{i}M_k$.\\
\textbf{If} \eqref{eq:adp} \text{holds} with $C_p^e = -g(-R_p)$ and $M = 2^{i}M_k$, then go to step 3.\\
\textbf{Else} set $i=i+1$ and go to Step 1.\\
\textbf{End If}
\item[3.] set $k=k+1$, $M_{k+1} = 2^{i-1}M_k$ and $i=0$.\\
\end{enumerate}
\textbf{End While}
}%
}
\end{center}
\vspace{0.2cm}
\noindent For a better understanding of this process, let us consider Example \ref{exam:1}, where $F = F_1 + F_2$, having  the $p$ derivative of $F_1$ $L_p^{F_1}$-Lipschitz and  $F_2$  proper closed convex function. Then, in this case  the surrogate is $s_M(y;\!\!\!x) = T^{F^1}_p(y;\!\!x) + \frac{M}{(p+1)!}\|y - x\|^{p+1} + F_2(y)$. Let $R_p,M_0 > 0$ be fixed. Then, the step 1 in A-GCHO algorithm can be seen as a line search procedure (see for example \cite{GraNes:19}): that is at each step $k \geq 0$ we choose $M_k\geq M_0$, then build $s_{M_k}(y;\!\!x_k) = T^{F^1}_p(y;\!\!x_k) + \frac{M_k}{(p+1)!}\|y - x_k\|^{p+1} + F_2(y)$ and compute $x_{k+1}$ satisfying \eqref{eq:adp0}. If \eqref{eq:adp} doesn't hold, then we increase  $M_k \leftarrow  2\cdot M_k$, recompute $s_{M_k}(y;\!\!x_k)$ using the new $M_k$ and go to step 2. We repeat this process until  condition \eqref{eq:adp} is satisfied. Note that this line search procedure finishes in a finite number of steps. Indeed, if $M_k \geq R_p + L_p^{F^1}$, then from inequality \eqref{eq:cd_adp}, we get $s_{M_k}(y;\!\!x_k) - F(y) \geq \frac{R_p}{(p+1)!}\|y - x_k\|^{p+1}$ for all $y$ and thus for $y = x_{k+1}$ and $g$ increasing function \eqref{eq:adp} holds. Note also that in this case the error function  $e$ satisfies  Definition \ref{def:surg} (i) with $L_p^e = 2(R_p + L_p^{F^1})$. Hence, using the same convergence analysis as in the previous sections, we can derive similar convergence rates as in Theorems  \ref{Th:ncx}, \ref{th:MHO} and \ref{Thm:kl} for A-GCHO algorithm under Assumption \ref{cond:2} and \ref{cond:3}, since the sequence $(x_k)_{k\geq 0}$ generated by A-GCHO automatically satisfies \eqref{eq:adp0} and \eqref{eq:adp}. For the convex case, as in Section 3.4, in A-GCHO algorithm we require that $x_{k+1}$ is the global solution of the corresponding subproblem and consequently similar convergence results as in Theorem \ref{Theorem:convex} can be derived for this adaptive algorithm.

\section{Numerical simulations}\label{sec4}
\red{In this section we present some preliminary numerical results for GCHO algorithm.  For simulations,  we consider the tests set from  \cite{MoGaHi:81}. In \cite{MoGaHi:81}, one can find systems of nonlinear equations, where one searches for $x^{*}$ such that $F_{i}(x^{*}) = 0$ for all  $i = 1,\cdots,m$. For solving these problems, we implement our GCHO algorithm for $p=1,2$. We consider two formulations:  min-max and least-squares problems, respectively. The min-max formulation has the form:}

\vspace{-0.6cm}

\begin{align}\label{min:2}
\min\limits_{x \in \mathbb{R}^n} f(x)   := \max(F_{1}^{2}(x),\cdots,F_{m}^{2}(x)).
\end{align}

\vspace{-0.3cm}

\noindent Similarly,  the least-squares formulation can be written as a simple composite minimization problem:

\vspace{-1cm}

\begin{align}\label{min:1}
\min\limits_{x \in  \mathbb{R}^n}   f(x) := \sum^{m}_{i=0} F_{i}^{2}(x).
\end{align} 

\vspace{-0.4cm}

\noindent Note that both formulation {fit} into our general problem \eqref{eq:optpb}. We consider the following 2 implementations. First, for problem \eqref{min:2}  we compare GCHO algorithm for $p=1,2$ with IPOPT (the results are given in Table \ref{tab2}). Secondly, for problem \eqref{min:1}  we compare GCHO algorithm  for $p=1,2$ with IPOPT and the method proposed in \cite{GoReSc:19} (the results are given in Table \ref{tab3}).  At each iteration of GCHO algorithm  we replace each function $F_{i}$ by its Taylor approximation of order $p$, with $p=1,2$, and a quadratic/cubic  regularization and solve the corresponding subproblem \eqref{eq:sub_prbm} using IPOPT \cite{WacBie:06}.  In the numerical simulations we have noticed that for $p=2$ IPOPT was able to detect a global minimizer of the subproblem at each iteration, i.e., the solution of IPOPT coincided with the solution obtained by solving the dual problem as described in the proof of Lemma \ref{Lem:3} given in the appendix.  Since it is difficult to compute the corresponding Lipschitz constants for the gradient/hessian, we use the line search procedure described in Section \ref{sec:ad}. Note that since in practice it is difficult to compute the sequence $(y_k)_{k\geq 1}$, then we cannot consider $\text{dist}(0,\partial f(y_k))\leq\epsilon$ as a stooping criterion for the proposed algorithm. \blue{Thus the stopping criterion considered in this paper is the same as in  \cite{BriGar:20}:
\begin{align*}
   \frac{f(x_k) - f_{\text{best}}}{\max(1,f_\text{best})} \leq 10^{-4}, 
\end{align*}
where  $f_{\text{best}} = f^* \approx 0$, but positive, and the starting point $x_0$ are taken from \cite{MoGaHi:81}}.  In Tables \ref{tab2} and \ref{tab3}, we summarize our numerical results in terms of cpu time and number of iterations for GCHO algorithm $p = 1,2$, IPOPT and \cite{GoReSc:19}. Note that the test functions we consider in the two tables are nonconvex and most of them satisfy the KL condition (as semi-algrabraic functions). From the tables, \red{we observe that GCHO algorithm ($p=1$ or $p=2$) applied to the min-max formulation performs better than the GCHO algorithm ($p=1$ or $p=2$) applied to the the least-squares problem, both in cpu time and  number of iterations. This is due to the fact that the regularization constants  for the min-max problem  \eqref{min:2}, $M_{p}^{\text{max}} = \big(M_{p}^{\text{max}}(1),  \cdots, M_{p}^{\text{max}}(m)\big)$, are much smaller than the one for the least-squares formulation \eqref{min:1},   $M_{p}^\text{ls}$, i.e.,  $M_{p}^\text{ls} \approx \sum_{i=1}^m M_{p}^{\text{max}} (i)$}. Moreover, from the tables we observe that increasing the order of the Taylor approximation is beneficial for the GCHO algorithm: e.g., in the min-max formulation, GCHO with $p=2$ is at least twice faster than GCHO with $p=1$. \red{We also observe from Table \ref{tab3} that GCHO algorithm applied to min-max formulation for $p=2$ has a better behavior (in both cpu time and number of iterations) than the method proposed in \cite{GoReSc:19} for the least-squares formulation.  Finally, GCHO algorithm ($p=1, 2$)  for both formulations is able to identify the global optimal points/values given in \cite{MoGaHi:81}, while IPOPT directly applied to the formulations \eqref{min:1} or \eqref{min:2} may fail to identify the global optimal points/values (see Tables 2 and 3).}
\begin{table}[]
\begin{tabular}{|l|ll|ll|ll|}

\hline
min-max formulation          & \multicolumn{2}{l|}{GCHO(p=1)}        & \multicolumn{2}{l|}{GCHO(p=2)}        & \multicolumn{2}{l|}{IPOPT for \eqref{min:2}} \\ \hline
test functions  & \multicolumn{1}{l|}{iter} & cpu & \multicolumn{1}{l|}{iter} & cpu& \multicolumn{1}{l|}{iter}  & cpu \\ \hline


(2) Fre        & \multicolumn{1}{l|}{\blue{32}}     & \blue{0.78}    & \multicolumn{1}{l|}{\blue{5}}     & \blue{0.18}    & \multicolumn{1}{l|}{85}      & 0.01    \\ \hline





(7) Hel    & \multicolumn{1}{l|}{\blue{33}}     & \blue{0.74}     & \multicolumn{1}{l|}{\blue{11}}     & \blue{0.4}     & \multicolumn{1}{l|}{29}      & 0.02    \\ \hline

(8) Bar     & \multicolumn{1}{l|}{\blue{19}}     & \blue{0.69}     & \multicolumn{1}{l|}{\blue{8}}     & \blue{0.42}    & \multicolumn{1}{l|}{27}      & 0.06$^*$      \\ \hline

(9) Gau     & \multicolumn{1}{l|}{\blue{9}}     & \blue{0.33}     & \multicolumn{1}{l|}{\blue{2}}     & \blue{0.15}    & \multicolumn{1}{l|}{16}      & 0.04$^*$     \\ \hline

(12) Box     & \multicolumn{1}{l|}{\blue{23}}     & \blue{0.9}     & \multicolumn{1}{l|}{\blue{9}}     & \blue{0.5}     & \multicolumn{1}{l|}{36}      & 0.08$^{*}$     \\ \hline



(15) Kow $(m=11,n=4)$    & \multicolumn{1}{l|}{\blue{48}}     & \blue{0.7}     & \multicolumn{1}{l|}{\blue{7}}     &  \blue{0.35}   & \multicolumn{1}{l|}{5000}      & 7.8$^*$  \\ \hline

(17) Osb-1$(m=33,n=5)$         & \multicolumn{1}{l|}{\blue{57}}      & \blue{3.8}  &  \multicolumn{1}{l|}{\blue{9}}     &  \blue{1.7}    &  \multicolumn{1}{l|}{40}      & 0.9    \\ \hline

(18) Big $(m=13,n=6)$     & \multicolumn{1}{l|}{\blue{149}}     & \blue{7.73}     & \multicolumn{1}{l|}{\blue{14}}     & \blue{0.6}     & \multicolumn{1}{l|}{593}      & 1.5$^*$     \\ \hline

(19) Osb-2 ($m\!=\!65,n\!=\!11$)         & \multicolumn{1}{l|}{\blue{67}}     & \blue{18.75}     & \multicolumn{1}{l|}{\blue{20}}     & \blue{12.1}     & \multicolumn{1}{l|}{55}      &3.7$^*$     \\ \hline

(20) Wat (m=31,n=9)      & \multicolumn{1}{l|}{\blue{23}}     &\blue{2.56}     & \multicolumn{1}{l|}{\blue{7}}     & \blue{2.63}    & \multicolumn{1}{l|}{5000}      & 50.5$^*$    \\ \hline


(21) E-Ros ($n=m=6$)      & \multicolumn{1}{l|}{\blue{21}}     & \blue{0.63}     & \multicolumn{1}{l|}{\blue{3}}     & \blue{0.21}    & \multicolumn{1}{l|}{379}      & 0.72     \\ \hline

(21) E-Ros ($n=m=20$)      & \multicolumn{1}{l|}{\blue{26}}     & \blue{1.7}    & \multicolumn{1}{l|}{\blue{3}}     & \blue{0.53}    & \multicolumn{1}{l|}{124}      & 3.8    \\ \hline

(21) E-Ros ($n=m\!=\!100$)     & \multicolumn{1}{l|}{\blue{25} }     &  \blue{102.5}   & \multicolumn{1}{l|}{\blue{5}}     & \blue{40.1}     & \multicolumn{1}{l|}{119}      &  133.9   \\ \hline

(24) Pen II ($n=10$)   & \multicolumn{1}{l|}{\blue{61}}     & \blue{6.4}     & \multicolumn{1}{l|}{\blue{3}}     &  \blue{0.32}    & \multicolumn{1}{l|}{64}      &0.9$^*$    \\ \hline


(26) Tri ($n=10$)    & \multicolumn{1}{l|}{\blue{20}}     & \blue{0.53}    & \multicolumn{1}{l|}{\blue{3}}     & \blue{0.22}      & \multicolumn{1}{l|}{45}      &0.2$^*$    \\ \hline

(30) Bro ($n=10$)     & \multicolumn{1}{l|}{\blue{44}}     & \blue{0.88}     & \multicolumn{1}{l|}{\blue{3}}     & \blue{0.25}     & \multicolumn{1}{l|}{118}      &0.3$^*$    \\ \hline

\end{tabular}
\caption{Behaviour of GCHO  for $p=1,2$ and IPOPT for the min-max formulation \eqref{min:2}. Here "*" means that IPOPT didn't find $x^*$/$f^*$ reported in  \cite{MoGaHi:81}.}\label{tab2}

\vspace{0.2cm}
\begin{tabular}{|l|ll|ll|ll|ll|}
\hline

L.S formulation & \multicolumn{2}{l|}{GCHO(p=1)}        & \multicolumn{2}{l|}{GCHO(p=2)}        & \multicolumn{2}{l|}{\cite{GoReSc:19}} & \multicolumn{2}{l|}{IPOPT for \eqref{min:1}}       \\ \hline

test functions     & \multicolumn{1}{l|}{iter} & cpu & \multicolumn{1}{l|}{iter} & cpu & \multicolumn{1}{l|}{iter}  & cpu & \multicolumn{1}{l|}{iter} & cpu \\ \hline

(2) Fre      & \multicolumn{1}{l|}{\blue{562}}     & \blue{7.2}    & \multicolumn{1}{l|}{\blue{23}}     & \blue{0.48}    & \multicolumn{1}{l|}{\blue{7}}      & \blue{0.19}    & \multicolumn{1}{l|}{85}     &0.06     \\ \hline

(7) Hel      & \multicolumn{1}{l|}{\blue{59}}     &\blue{1.2}    & \multicolumn{1}{l|}{\blue{25}}     & \blue{0.95}    & \multicolumn{1}{l|}{\blue{15}}     & \blue{0.55}     & \multicolumn{1}{l|}{12}     &0.02     \\ \hline

(8) Bar       & \multicolumn{1}{l|}{\blue{88}}     & \blue{1.3}     & \multicolumn{1}{l|}{\blue{13}}     & \blue{0.5}    & \multicolumn{1}{l|}{\blue{12}}      & \blue{0.48}    & \multicolumn{1}{l|}{26}     & 0.04    \\ \hline

(9) Gau      & \multicolumn{1}{l|}{\blue{71}}     & \blue{1.25}    & \multicolumn{1}{l|}{\blue{13}}     & \blue{0.65}    & \multicolumn{1}{l|}{\blue{5}}     & \blue{0.17}     & \multicolumn{1}{l|}{8}     &0.03$^*$     \\ \hline

(12) Box       & \multicolumn{1}{l|}{\blue{719}}     & \blue{12.1}    & \multicolumn{1}{l|}{\blue{51}}     & \blue{2.05}     & \multicolumn{1}{l|}{\blue{13}}      & \blue{0.68}     & \multicolumn{1}{l|}{34}     & 0.05    \\ \hline



(15) Kow   & \multicolumn{1}{l|}{\blue{534}}     &  \blue{13.1}    & \multicolumn{1}{l|}{\blue{14}}     & \blue{0.67}     & \multicolumn{1}{l|}{\blue{10}}      & \blue{0.49}    & \multicolumn{1}{l|}{825}     & 1.98$^*$    \\ \hline

(17) Osb-1   & \multicolumn{1}{l|}{\blue{815}}     &  \blue{45.8}   & \multicolumn{1}{l|}{\blue{101}}     & \blue{9.6}     & \multicolumn{1}{l|}{\blue{18}} &  \blue{3.6}     & \multicolumn{1}{l|}{103}     & 1.9    \\ \hline

(18) Big      & \multicolumn{1}{l|}{\blue{968}}     & \blue{18.5}    & \multicolumn{1}{l|}{\blue{44}}     & \blue{2.19}    & \multicolumn{1}{l|}{\blue{17}}      & \blue{0.79}    & \multicolumn{1}{l|}{44}     & 0.15$^*$     \\ \hline

(19) Osb-2        & \multicolumn{1}{l|}{\blue{365}}     & \blue{45.9}    & \multicolumn{1}{l|}{\blue{82}}     & \blue{35.6}     & \multicolumn{1}{l|}{\blue{29}}      & \blue{15.3}    & \multicolumn{1}{l|}{329}     &  11.5$^*$    \\ \hline

(20) Wat       & \multicolumn{1}{l|}{\blue{161}}     &  \blue{50.6}    & \multicolumn{1}{l|}{\blue{21}}     & \blue{7.6}    & \multicolumn{1}{l|}{\blue{10}}      & \blue{3.66}   & \multicolumn{1}{l|}{794}     &  8.16$^*$    \\ \hline

(21) E-Ros       & \multicolumn{1}{l|}{\blue{2563}}     & \blue{38.7}     & \multicolumn{1}{l|}{\blue{12}}     & \blue{0.93}    & \multicolumn{1}{l|}{\blue{4}}      & \blue{0.28}    & \multicolumn{1}{l|}{83}     & 0.33     \\ \hline

(21) E-Ros       & \multicolumn{1}{l|}{ \blue{3040}}     & \blue{82.3}    & \multicolumn{1}{l|}{\blue{28}}     & \blue{9.4}     & \multicolumn{1}{l|}{\blue{5}}      & \blue{1.53}    & \multicolumn{1}{l|}{233}     &  1.8    \\ \hline

(21) E-Ros       & \multicolumn{1}{l|}{\blue{530}}     & \blue{253}    & \multicolumn{1}{l|}{\blue{33}}     &\blue{288.2}     & \multicolumn{1}{l|}{\blue{7}}      & \blue{71.5}    & \multicolumn{1}{l|}{223}     &  162.4 \\ \hline

(24) Pen II      & \multicolumn{1}{l|}{\blue{147}}     & \blue{10.2}     & \multicolumn{1}{l|}{\blue{7}}     & \blue{0.8}     & \multicolumn{1}{l|}{\blue{3}}      & \blue{0.35}     & \multicolumn{1}{l|}{22}     & 0.08$^*$  \\ \hline


(26) Tri       & \multicolumn{1}{l|}{\blue{28}}     & \blue{0.55}    & \multicolumn{1}{l|}{\blue{5}}     & \blue{0.3}     & \multicolumn{1}{l|}{\blue{3}}      & \blue{0.22}    & \multicolumn{1}{l|}{26}     &0.05$^*$   \\ \hline

(30) Bro        & \multicolumn{1}{l|}{\blue{56}}     & \blue{0.9}    & \multicolumn{1}{l|}{\blue{12}}     & \blue{0.59}     & \multicolumn{1}{l|}{\blue{4}}      & \blue{0.35}    & \multicolumn{1}{l|}{36}     &0.07$^*$   \\ \hline

\end{tabular}
\caption{Behaviour of GCHO algorithm for $p=1,2$, algorithm \cite{GoReSc:19} and IPOPT for the least-squares problem \eqref{min:1}. Here "*" means that IPOPT didn't find $x^*$/$f^*$ reported in  \cite{MoGaHi:81}. }\label{tab3}
\end{table}

\vspace{-0.2cm}


\section{Appendix}
\begin{appendices}

\begin{proof}\textit{of Lemma \ref{Lem:3}}
Let us first prove that for $p=2$,  $g(\cdot)=\max(\cdot)$ \blue{and $h(\cdot) = 0$}, one can compute efficiently the global solution  $x_{k+1}$ of the subproblem \eqref{eq:sub_prbm}. Indeed,  in this particular case \eqref{eq:sub_prbm} is equivalent to the following subproblem:
\begin{align}\tag{\blue{39}}
\min_{x \in \mathbb{R}^n} \max\limits_{i=1:m}&\left\{ F_i(x_k) + \langle \nabla F_i(x_k), x - x_k \rangle + \frac{1}{2} \left\langle \nabla^2 F_i(x_k)(x - x_k), x - x_k\right\rangle\right.\\ \nonumber 
& +\left. \frac{M_i}{6} \|x - x_k \|^3\right\}.
\end{align}
Further, this is equivalent to:
\begin{align*}
    \min_{x \in \mathbb{R}^n} \max_{\substack{u \in\Delta_m}}\; &\sum_{i=1}^{m} u_i F_i(x_k) + \left\langle \sum_{i=1}^{m}u_i \nabla F_i(x_k), x - x_k \right\rangle \\
    & + \frac{1}{2} \left\langle \sum_{i=1}^{m} u_i \nabla^2 F_i(x_k)(x - x_k), x - x_k \right\rangle+ \frac{\sum_{i=1}^{m}u_i M_i}{6} \|x - x_k\|^3,
\end{align*}
where $u=(u_1,\cdots,u_m)$ and $\Delta_m := \left\{u\geq 0 : \sum_{i=1}^{m} u_i = 1 \right\} $ is the standard simplex in $\mathbb{R}^m$. Further, this $\min-\max$ problem can be written as follows:
\begin{align*}
&\min_{x \in \mathbb{R}^n} \max_{ \substack{u\in\Delta_M}}\; \sum_{i=1}^{m} u_i F_i(x_k) + \left\langle \sum_{i=1}^{m}u_i \nabla F_i(x_k), x \!-\! x_k \right\rangle\\
& + \frac{1}{2} \left\langle \sum_{i=1}^{m} u_i \nabla^2 F_i(x_k) (x \!-\! x_k), (x \!-\! x_k) \right\rangle
     \!+\! \max_{w\geq 0} \left( \frac{w}{4}\|x \!-\! x_k\|^2 \!-\! \frac{1}{12(\sum_{i=1}^{m} u_i M_i)^2} w^3 \!\!\right).
\end{align*}
Denote for simplicity $H_k(u,w) = \sum_{i=1}^{m} u_i \nabla^2 F_i(x_k) + \frac{w}{2} I$, $g_k(u) = \sum_{i=1}^{m}u_i \nabla F_i(x_k)$, $l_k(u) = \sum_{i=1}^{m} u_i F_i(x_k)$ and $\Tilde{M}(u) = \sum_{i=1}^{m}u_i M_i$. 
Then, the dual formulation of this problem takes the form:
\begin{align*}
 \min_{x \in \mathbb{R}^n} \max_{ \substack{u\in\Delta_m \\ w\in\mathbb{R}_+}}\; l_k(u) + \left\langle g_k(u), x \!-\! x_k \right\rangle \!+\! \frac{1}{2} \left\langle H_k(u,w) (x \!-\! x_k), (x \!-\! x_k) \right\rangle
  - \frac{w^3}{12\Tilde{M}(u)^2}.    
\end{align*}
Consider the following notations:
\begin{align*}
&  \theta(x,u) = l_k(u) + \langle  g_k(u), x - x_k \rangle \!+\! \frac{1}{2}\! \left\langle\!\! \left(\sum_{i=1}^{m} u_i \nabla^2 F_i(x_k)\right)\! (x \!-\! x_k), x \!-\! x_k \!\!\right\rangle\\
&\qquad\qquad + \frac{\Tilde{M}(u)}{6}\|x \!-\! x_k\|^3, \\
    & \beta(u,w) = l_k(u) -\frac{1}{2} \left\langle H_k(u,w)^{-1} g(u) , g(u)\right\rangle - \frac{1}{12\Tilde{M}(u)^2} w^3, \\
    & D = \left\{ (u,w)\in\Delta_m\times\mathbb{R}_+:\; \; \text{s.t.}\; \sum_{i=1}^{m} u_i \nabla^2 F_i(x_k) + \frac{w}{2} I\succ 0  \right\}.
\end{align*}
Below, we prove that if there exists an $M_i >0$, for some $i = 1:m$, then we have the following relation:
    \begin{align*}
        \theta^*:=\min_{x\in\mathbb{R}^n}\max_{u\in \Delta_m} \theta(x,u)  =  \max_{(u,w)\in D}\beta(u,w) = \beta^*. 
    \end{align*}
\noindent Additionally, for any $(u,w)\in D$  the  direction $x_{k+1} = x_k -H_k(u,w)^{-1}g_k(u)$ satisfies:
 \begin{align}\tag{\blue{40}}
     0\leq \theta(x_{k+1},u) - \beta(u,w)  = \frac{\Tilde{M}(u)}{12} \left(\frac{w}{\Tilde{M}(u)} + 2r_k\right)\left(r_k - \frac{w}{\Tilde{M}(u)}\right)^2,
 \end{align}
where $r_k := \|x_{k+1} - x_k\|$.
Indeed, let us first show that $\theta^*\geq \beta^*$. Using a similar reasoning as in \cite{NesPol:06}, we have:
\begin{align*}
\theta^* & =  \min_{x \in \mathbb{R}^n} \max_{ \substack{u\in\Delta_m\\ w  \in\mathbb{R}_+}}\; l_k(u) + \left\langle g_k(u), x - x_k \right\rangle + \frac{1}{2} \left\langle H_k(u,w) (x - x_k), x - x_k\right\rangle \!-\! \frac{ w^3}{12\Tilde{M}(u)^2} \\
     &\geq \max_{ \substack{u\in\Delta_m\\ w  \in\mathbb{R}_+}} \min_{x \in \mathbb{R}^n}\;l_k(u) + \left\langle g_k(u), x - x_k \right\rangle + \frac{1}{2} \left\langle H_k(u,w) (x - x_k), x - x_k \right\rangle\!-\! \frac{ w^3}{12\Tilde{M}(u)^2} \\
     &\geq \max_{(u,w)\in D} \min_{x \in \mathbb{R}^n}\;l_k(u) \!+\! \left\langle g_k(u), x - x_k \right\rangle \!+\! \frac{1}{2} \left\langle H_k(u,w) (x - x_k), x - x_k \right\rangle \!-\! \frac{ w^3}{12\Tilde{M}(u)^2}\\
     & = \max_{(u,w)\in D} \blue{l_k(u)} -\frac{1}{2} \left\langle H_k(u,w)^{-1} g_k(u) , g_k(u)\right\rangle - \frac{1}{12\Tilde{M}(u)^2} w^3
      = \beta^*.
\end{align*}
Let $(u,w) \in D$. Then, we have $g_k(u)= - H_k(u,w)(x_{k+1} - x_k)$ and thus: 
\begin{align*}
&\theta(x_{k+1},u) = l_k(u) +   \langle  g_k(u), x_{k+1} - x_k \rangle  \\
&\qquad + \frac{1}{2} \left\langle \left(\sum_{i=1}^{m} u_i \nabla^2 F_i(x_k)\right) (x_{k+1}- x_k), x_{k+1} - x_k \right\rangle +  \frac{\Tilde{M}(u)}{6}r_k^3 \\
& = l_k(u)  - \left\langle H_k(u,w)(x_{k+1} - x_k), x_{k+1} - x_k \right\rangle \\
&\qquad + \frac{1}{2} \left\langle \left(\sum_{i=1}^{m} u_i \nabla^2 F_i(x_k)\right) (x_{k+1}- x_k), x_{k+1} - x_k \right\rangle + \frac{\Tilde{M}(u)}{6}r_k^3  \\  
& = l_k(u) \!-\!\frac{1}{2}\!\! \left\langle \!\! \left(\sum_{i=1}^{m} u_i \nabla^2 F_i(x_k) \!+\! \frac{w}{2}I\right)(x_{k+1} \!-\! x_k), x_{k+1} \!-\! x_k \!\!\right\rangle \!-\! \frac{w}{4}r_k^2 \!+\!  \frac{\Tilde{M}(u)}{6}r_k^3 \\
& = \beta(u,w) + \frac{1}{12\Tilde{M}(u)^2}w^3 - \frac{w}{4}r_k^2 +  \frac{\Tilde{M}(u)}{6}r_k^3 \\
& = \beta(u,w) + \frac{\Tilde{M}(u)}{12}\!\left(\frac{w}{\Tilde{M}(u)}\right)^3 \!\!\!- \!\frac{\Tilde{M}(u)}{4}\!\! \left(\frac{w}{\Tilde{M}(u)}\right)\!r_k^2\!+\! \frac{\Tilde{M}(u)}{6}r_k^3 \\
& =  \beta(u,w) + \frac{\Tilde{M}(u)}{12} \left(\frac{w}{\Tilde{M}(u)} + 2r_k\right)\left(r_k - \frac{w}{\Tilde{M}(u)}\right)^2,
\end{align*}
which proves (\blue{40}). Note that we have \cite{NesPol:06}:
\begin{align*}
    \nabla_w \beta(u,w) & = \frac{1}{4} \|x_{k+1} - x_k\|^2 - \frac{1}{4\Tilde{M}(u)^2} w^2  = \frac{1}{4} \left(r_k + \frac{w}{\Tilde{M}(u)}\right)\left(r_k - \frac{w}{\Tilde{M}(u)}\right).
\end{align*}
Therefore, if $\beta^*$ is attained at some $(u^*,w^*) \in D$, then we have $\nabla \beta(u^*,w^*) = 0$. This implies $\frac{w^*}{\Tilde{M}(u^*)} = r_k$ and by (\blue{40}) we conclude that $\theta^* = \beta^*$.\\
Finally, if $x_{k+1}$ is a global solution of the subproblem \eqref{eq:sub_prbm} (or equivalently (\blue{39})), then it satisfies the inexact condition \eqref{eq:inx} with $\delta = 0$. Hence, using the proof of Lemma \ref{res:lem} with $\delta = 0$ we can conclude that Assumption \ref{cond:2} holds with $y_{k+1}$ given in \eqref{stp:prx}, $L^{1}_{p}=\left(C_{L^{e}_{p}}^{\mu_{p}}\right)^{1/3}$ and $L^{2}_{p}=\frac{\mu_{p}}{2}$. 
\end{proof}


\begin{proof}\!\textit{of Remark}  \ref{remark:a3}. 
If $g$ is the identity function, then taking $y_{k+1}=x_{k+1}$ one can see that  Assumption \ref{cond:3} holds for any $\theta_{1,p}$ and $\theta_{2,p}$  nonnegative constants.  If $g$ is a general function, then Assumption \ref{cond:3} holds,  provided that $x_{k+1}$ satisfies the inexact optimality condition \eqref{eq:inx}. Indeed, in this case, we have:
\begin{align*}
f(x_{k+1}) &\leq g \Big(s(x_{k+1};\!\!x_{k})\Big) + h(x_{k+1})\\ 
           &\stackrel{\eqref{eq:inx}}{\leq} \min_{\red{y:\; \|y - x_k\|\leq D_k}} g \Big(s(y;\!\!x_{k})\Big) + h(y) + \delta\norm{x_{k+1} - x_k}^{p+1}\\
           &\stackrel{\eqref{eq:qq1},\eqref{Lemma:1}}{\leq} \min_{\red{y:\;\|y - x_k\|\leq D_k}} g\big(F(y)\big)+ h(y) + \dfrac{g(L^{e}_{p})}{(p+1)!}\norm{y-x_{k}}^{p+1}\\ &\qquad + \delta\norm{x_{k+1} - x_k}^{p+1}\\
           &\leq f(y_{k+1}) + \dfrac{g(L^{e}_{p})}{(p+1)!}\norm{y_{k+1}-x_{k}}^{p+1} + \delta\norm{x_{k+1} - x_k}^{p+1},      
\end{align*}
where the last inequality follows  taking $y = y_{k+1}$. Hence, Assumption \red{\ref{cond:3} holds} in this case for $\theta_{1,p} =  \dfrac{g(L^{e}_{p})}{(p+1)!}$ and $\theta_{2,p} = \delta$. Finally, if $p=2$ and $g(\cdot) = \max(\cdot)$, then $x_{k+1}$ is the global solution of the subproblem \eqref{eq:sub_prbm} and hence,  using similar arguments as above, we can prove that Assumption \ref{cond:3} also holds  in this case.
\end{proof}


\begin{proof}\!\textit{of Lemma} \ref{lemcv}. 
Note that the sequence $\lambda_{k}$ is nonincreasing and nonnegative, thus it is convergent.  Let us consider first  $\theta \leq 1$. Since  $\lambda_{k}-\lambda_{k+1}$ converges to $0$,  then there exists $k_{0}$ such that $\lambda_{k}-\lambda_{k+1} \leq 1$ and  $\lambda_{k+1}\leq (C_{1}+C_{2})\left(\lambda_{k}-\lambda_{k+1}\right)$ for all $k \geq k_{0}$.
It follows that:

\vspace{-0.5cm}

\begin{align*}
\lambda_{k+1}\leq \frac{C_{1}+C_{2}}{1+C_{1}+C_{2}}\lambda_{k},
\end{align*}
which proves the first statement.  If $1<\theta \leq 2$, then there exists also an integer $k_{0}$ such that  $\lambda_{k}-\lambda_{k+1} \leq 1$  for all $k\geq k_{0}$. Then, we have:
\begin{align*}
\lambda_{k+1}^{\theta}\leq (C_{1}+C_{2})^{\theta}\left(\lambda_{k}-\lambda_{k+1}\right).
\end{align*}
Since $1<\theta \leq 2$, then  taking  $0 < \beta = \theta -1 \leq 1$, we have:
\begin{align*}
\left(\frac{1}{C_{1}+C_{2}}\right)^{\theta}\lambda_{k+1}^{1+\beta}\leq \lambda_{k}-\lambda_{k+1},
\end{align*}
for all $k\geq k_{0}$. From Lemma 11 in \cite{Nes:20Inxt}, we further have:
\begin{align*}
\lambda_{k}\leq \frac{\lambda_{k_{0}}}{(1+\sigma (k-k_{0}))^{\frac{1}{\beta}}}
\end{align*}
for all $k\geq k_{0}$ and for some $\sigma >0$. Finally, if $\theta > 2$, then define $h(s)\blue{=}s^{-\theta}$ and let $R>1$ be fixed. Since $1/\theta < 1$, then there {exists} a $k_{0}$ such that $\lambda_{k}-\lambda_{k+1} \leq 1$ for all $k \geq k_{0}$. Then,  we have $\lambda_{k+1}\leq (C_{1}+C_{2})\left(\lambda_{k}-\lambda_{k+1}\right)^{\frac{1}{\theta}}$, or equivalently:
\begin{align*}
1\leq (C_{1}+C_{2})^{\theta}(\lambda_{k}-\lambda_{k+1})h(\lambda_{k+1}).
\end{align*}
If we assume that  $h(\lambda_{k+1})\leq R h(\lambda_{k}) $, then:
\begin{align*}
1\leq R(C_{1}+C_{2})^{\theta}(\lambda_{k}-\lambda_{k+1})h(\lambda_{k})&\leq 
\frac{R(C_{1}+C_{2})^{\theta}}{-\theta+1}\left(\lambda_{k}^{-\theta+1}-\lambda_{k+1}^{-\theta+1} \right).
\end{align*}
Denote $\mu=\frac{-R(C_{1}+C_{2})^{\theta}}{-\theta+1}$. Then:
\begin{align}\label{eq:recc1}
0<\mu^{-1} \leq \lambda_{k+1}^{1-\theta}  -  \lambda_{k}^{1-\theta}.
\end{align}
If we assume that $h(\lambda_{k+1})> R h(\lambda_{k}) $ and set $\gamma = R^{-\frac{1}{\theta}}$, then it follows immediately that $\lambda_{k+1}\leq \gamma \lambda_{k}$. Since $1-\theta$ is negative, we get:
\begin{align*}
\lambda_{k+1}^{1-\theta}&\geq \gamma^{1-\theta}\lambda_{k}^{1-\theta} \quad \iff \quad 
\lambda_{k+1}^{1-\theta} - \lambda_{k}^{1-\theta} \geq (\gamma^{1-\theta}-1)\lambda_{k}^{1-\theta}.
\end{align*}
Since $1- \theta<0$,  $\gamma^{1-\theta} > 1$ and $\lambda_{k}$ has a nonnegative limit, then there exists $\bar{\mu} > 0$ such that $(\gamma^{1-\theta} - 1) \lambda_{k}^{1-\theta} > \bar{\mu}$ for all $k \geq k_0$. Therefore,  in this case we also obtain:
\begin{align}\label{eq:recc2}
0<  \bar{\mu} \leq \lambda_{k+1}^{1-\theta}-\lambda_{k}^{1-\theta}.
\end{align}
If we set $\hat{\mu}=\min(\mu^{-1},\bar{\mu})$ and combine \eqref{eq:recc1} and \eqref{eq:recc2}, we  obtain:
\begin{align*}
0<  \hat{\mu} \leq  \lambda_{k+1}^{1-\theta}-\lambda_{k}^{1-\theta}.
\end{align*}
Summing the last  inequality from $k_{0}$ to $k$, we obtain $\lambda_{k}^{1-\theta}-\lambda_{k_{0}}^{1-\theta}\geq \hat{\mu}(k-k_{0})$, i.e.:

\vspace{-0.3cm}

\begin{center}
$\lambda_{k} \leq   \frac{\hat{\mu}^{-\frac{1}{\theta-1}}}{(k-k_{0})^{\frac{1}{\theta-1}}}$
\end{center}

\vspace{-0.3cm}

\noindent for all $k \geq k_0$. This concludes our proof. 
\end{proof}


\end{appendices}


\section*{Acknowledgements}
\noindent The research leading to these results has received funding from: ITN-ETN project TraDE-OPT funded by the EU, H2020 Research and Innovation Programme under the Marie Skolodowska-Curie grant agreement No. 861137;  NO Grants 2014–2021, under project ELO-Hyp, contract no. 24/2020; UEFISCDI PN-III-P4-PCE-2021-0720, under project L2O-MOC, nr. 70/2022.\\\\
\noindent \textbf{Data availability} Not applicable.\\\\
\noindent\textbf{Conflicts of interest} The authors declare that they have no conflict of interest.


\end{document}